\documentclass[12pt]{article}
\usepackage[margin=0.6in]{geometry}
\usepackage{amsmath,amssymb,amsthm,mathtools}
\usepackage{bbm}
\usepackage[colorlinks=true,linkcolor=blue,citecolor=blue,urlcolor=blue]{hyperref}
\newtheorem{theorem}{Theorem}[section]
\newtheorem{proposition}[theorem]{Proposition}
\newtheorem{corollary}[theorem]{Corollary}
\newtheorem{lemma}[theorem]{Lemma}
\theoremstyle{definition}
\newtheorem{definition}{Definition}
\newtheorem{openproblem}{Open problem}
\theoremstyle{remark}

\numberwithin{equation}{section}

\def\<{\langle}\def\>{\rangle}

\newcommand{\R}{\mathbb R}
\newcommand{\N}{\mathbb N}
\newcommand{\Z}{\mathbb Z}
\newcommand{\E}{\mathbb E}
\newcommand{\Pp}{\mathbb P}
\newcommand{\Var}{\textrm{\rm var}}
\newcommand{\Cov}{\textrm{\rm cov}}
\newcommand{\dimH}{\dim_{\rm H}}
\newcommand{\Graph}{\textrm{\rm graph}}

\newcommand{\1}{\mathbbm 1}
\newcommand{\calH}{\mathcal H_H}
\newcommand{\dd}{\mathrm d}
\newcommand{\fp}[1]{\{#1\}}
\newcommand{\abs}[1]{\left|#1\right|}
\newcommand{\norm}[1]{\left\|#1\right\|}
\newcommand{\floor}[1]{\left\lfloor#1\right\rfloor}

\begin{document}

\title{Variation and Hausdorff dimension of fractional Wiener--Weierstrass bridges}
\author{\normalsize Alexander Schied\thanks{Department of Statistics and Actuarial Science, University of Waterloo. E-mail: {\tt aschied@uwaterloo.ca}}
	 \and\setcounter{footnote}{6}
\normalsize	Zhenyuan Zhang\thanks{
		 Laboratory for Information and Decision Systems, Massachusetts Institute of Technology. E-mail: {\tt zzyzzy@mit.edu}
	 		\hfill\break The first author gratefully acknowledges financial support  from the
 Natural Sciences and Engineering Research Council of Canada through grant  RGPIN-2024-03761.
The second author was supported by the Jump Trading Fellowship. The authors used ChatGPT 5.5 for language editing and mathematical exploration. The authors wrote the full text and take full responsibility for the content of the manuscript. }}
\date{\normalsize First version: June 10, 2026\\
This version: October 14, 2026}
\maketitle

\begin{abstract}
Fractional Wiener--Weierstrass bridges are a class of Gaussian processes obtained by replacing trigonometric functions in the construction of classical Weierstrass functions by fractional Brownian bridges. A number of their sample path properties were derived in Schied--Zhang (2024,\,2026). The analysis in these papers left several open questions, most of which are addressed here. Specifically, we prove that, in the regime in which the Weierstrass mechanism dominates the underlying fractional Brownian bridge, the limiting $b$-adic variation coefficient has an absolutely continuous distribution and is therefore genuinely random. At the critical point between the two roughness regimes, we establish the power-variation formula and the critical $\Phi$-variation limit conjectured in Schied--Zhang (2024). Finally, we derive the Hausdorff dimension for the graphs of the sample paths by proving a conjecture from Schied--Zhang (2026) for the missing high-Hurst case.
\end{abstract}

\noindent\textit{Keywords:} Fractional Wiener--Weierstrass bridge; Gaussian process; $p^{\text{th}}$ variation; $\Phi$-variation; Hausdorff
dimension.

\medskip
\noindent\textit{2020 Mathematics Subject Classification:} 60G22, 60G15,
60G17, 28A80, 60F15.

\section{Introduction}\label{sec:intro}

Interest in  stochastic processes whose sample paths may be rougher or smoother than Brownian motion,  has grown substantially in recent
years.  Several factors have contributed to this development. For instance, progress in areas such as  rough path theory, models based on fractional Brownian motion, and pathwise It\^o calculus for trajectories with $p^{\text{th}}$ variation for $p>2$ as in \cite{ContPerkowski2019} has stimulated further study. At the same time, practical applications, such as rough volatility modeling \cite{bayer2023rough}, have  been emphasizing their real-world importance.

Roughness is also a central theme in fractal
geometry.  The classical Weierstrass function and its variants provide basic
examples of continuous functions with highly irregular graphs, and their moduli
of continuity, variation, and Hausdorff dimensions have been studied by a range
of deterministic and probabilistic methods \cite{BaranskiBaranyRomanowska2014,Gamkrelidze1984,Hardy1916,SchiedZZhang,Shen2018}. Weierstrass-type fractal functions $f$ are usually defined as
\begin{equation}\label{eq:f-intro-eq}
f(t):=\sum_{n=0}^\infty \alpha^n\phi(\fp{b^n t}),\qquad 0\le t\le 1,
\end{equation}
where $\alpha\in(0,1)$, $b\in\{2,3,\dots\}$, $\phi: [0,1]\to\mathbb R$ is a Lipschitz continuous function with $\phi(0)=\phi(1)$, and~$\fp{x}$ denotes the fractional part of $x\ge0$.

Fractional Wiener--Weierstrass bridges, as introduced in  \cite{SchiedZhang2024}, combine  two sources of
fractality by replacing the Lipschitz function $\phi$ in the Weierstrass construction \eqref{eq:f-intro-eq} with the sample paths of a fractional Brownian bridge $B_H$ with Hurst parameter $H$. This construction gives rise to a Gaussian stochastic process $Y$ of the form 
\[   Y(t)=\sum_{m=0}^\infty \alpha^m B_H(\{b^m t\}),\qquad 0\le t\le1,
\]
called the \emph{fractional Wiener--Weierstrass bridge}. The stochastic process $Y$ has highly non-stationary increments, is not self-similar, and derives its
roughness from the competition between the Hurst exponent $H$ and the
roughness exponent $K:=\min\{1,-\log_b\alpha\}$ of the classical Weierstrass function~\eqref{eq:f-intro-eq}: As a matter of fact, it follows from the results in \cite{SchiedZhang2024,SchiedZhang2026} that the sample path properties of $Y$ have distinct qualitative features in the two regimes $K<H$ and $K>H$, and that additional phenomena arise in the critical case $K=H$. 

In \cite{SchiedZhang2024}, the authors identified the $b$-adic variation
exponents of $Y$ for $H\neq K$, proved that Weierstrass bridges are never semimartingales, and
showed that their covariance functions can themselves have a fractal structure. 
The  subsequent paper \cite{SchiedZhang2026} studies Wiener--Young
$\Phi$-variation, local and uniform moduli of continuity, Hausdorff dimension,
nowhere differentiability, and the distribution of the maximum location. 
The outlook sections of \cite{SchiedZhang2024} and \cite{SchiedZhang2026}
formulated several open questions, which could not be settled at the time, and formulated corresponding conjectures. 

The
present paper addresses the conjectures summarized below. The first open problem is based on \cite[Theorem 2.3]{SchiedZhang2024}, which states  that the process $Y$ has a deterministic linear $(1/H)$-variation in the regime
$H<K$, while for $H>K$ the limiting $(1/K)$-variation is
$Vt$, where $V$ is a finite positive random variable. Recall that for $f:[0,1]\to\R$ and $p\geq 1$, the $p$-th variation of $f$ (on the sequence of $b$-adic partitions, where $b\geq 2$) is defined as 
\begin{equation}\label{pth variation}
\lim_{n\uparrow\infty}\sum_{k=0}^{\lfloor tb^n\rfloor}\big|f((k+1)b^{-n})-f(kb^{-n})\big|^p,\qquad t\in[0,1],
\end{equation}
provided that the limit exists for all $t$.

\begin{openproblem}\label{op:nondegeneracy}
In the regime $H>K$, decide whether the random coefficient $V$ in the
limiting $(1/K)$-variation is non-degenerate, or whether it is in fact a
deterministic constant.
\end{openproblem}

Based on numerical experiments, it was conjectured in \cite{SchiedZhang2024}  that $V$  is a non-degenerate random variable and not merely a constant in disguise. Theorem~\ref{thm:op1}, our first main result, confirms this conjecture by proving
more than non-degeneracy: the random variable $V$ has an absolutely
continuous distribution and all finite moments.

The next two open problems concern the $b$-adic variation of $Y$ in the critical case $H=K$. 

\begin{openproblem}\label{op:critical-variation}
In the critical case $H=K$, find a normalizing sequence $(a_n)_{n\in\mathbb N}$ such that 
\begin{equation*}
 a_n\sum_{k=0}^{\floor{tb^n}-1}   |Y((k+1)b^{-n})-Y(kb^{-n})|^{1/H}
\end{equation*}
converges almost surely  to a finite limit for all $t\in[0,1]$. 
\end{openproblem}

In \cite[Theorem 2.4]{SchiedZhang2024}, the preceding problem was only solved in the Brownian case $H=K=1/2$ by showing that the $b$-adic quadratic variations, if rescaled by $a_n=1/n$, converge almost surely to the linear function $t$. For general $H$, it was conjectured that the correct scaling is $a_n=n^{-1/(2H)}$ and that the limit should be given by $2^{1/(2H)}\Gamma((1+H)/2H)t/\sqrt{\pi}$.  Theorem \ref{thm:critical-power} in the present paper now proves this conjecture. Since $\lim_na_n=0$, it follows in particular that the unscaled $p^{\text{th}}$-variation of $Y$ is infinite for $p=1/H=1/K$. This behavior clearly distinguishes the critical case from the case $H\neq K$, for which \cite[Theorem 2.3]{SchiedZhang2024} establishes the almost-sure convergence of the unscaled $p^{\text{th}}$-variation to a finite limit for $p=1/(H\wedge K)$. This led to the question whether the critical variation should be expressed through Wiener--Young-type $\Phi$-variation, where the function $|\cdot|^p$ is replaced by $\Phi(|\cdot|)$ in \eqref{pth variation}:

\begin{openproblem}\label{op:critical-Phi-variation} In the critical case $H=K$, identify a suitable function $\Phi$ such that the  $b$-adic  Wiener--Young-type $\Phi$-variation converges  almost surely to a finite limit. 
\end{openproblem}

It was conjectured in \cite{SchiedZhang2024} that one can take $\Phi(x)=x^{1/H}(-\log x)^{-1/(2H)}$ and that the limiting $\Phi$-variation is of the form $ {2^{1/(2H)}\Gamma(\frac{H+1}{2H})}
      (\pi\,(-\log\alpha)^{1/H})^{-1/2}\cdot t$. This conjecture is now proved in Theorem~\ref{thm:critical-phi}. As a corollary, we conclude that the roughness exponent of the sample paths of $Y$ is almost surely given by $H\wedge K$.
      
      Our next open problem is based on a conjecture from \cite{SchiedZhang2026} and now solved in Theorem~\ref{thm:hausdorff-high-hurst}.

\begin{openproblem}\label{op:hausdorff}
Prove that the Hausdorff dimension of the graph of $Y$ 
in the high-Hurst range $H>K\vee(1/2)$, $K\le2H-1$ is given by $\max\{2-H,2-K\}$. \end{openproblem}

Combined with \cite[Theorem 2.6]{SchiedZhang2026}, Theorem~\ref{thm:hausdorff-high-hurst} yields that the Hausdorff dimension of the graph of $Y$ is  almost surely given by   $\max\{2-H,2-K\}$ for all values of $H$ and $K$. This result is now formulated 
in Corollary \ref{cor:full}.

The proofs of our results rely on  a few complementary mechanisms summarized below. 
\begin{itemize}
    \item 
For $H>K$, a \lq\lq randomized-digits" representation from \cite{SchiedZZhang,SchiedZhang2024} expresses the  coefficient $V$ as an average of powers of Gaussian
linear functionals.  After projecting the Gaussian Hilbert space onto a suitable
one-dimensional direction, $V$ becomes, conditionally on the orthogonal
Gaussian field, a strictly convex $C^1$-function of a single standard normal variable;
this proves absolute continuity of its distribution.

 \item For the critical case $H=K$, the almost-sure convergence of the rescaled power variation is proved
by a second-moment method.  The expectation is computed from the randomized-digits representation, while the variance is controlled through two-scale
fractional Brownian covariance estimates, a Frobenius norm bound, and a Hermite
expansion estimate for powers of correlated Gaussian variables. 

 \item The
$\Phi$-variation statement follows by combining the convergence of the rescaled power variation  with a uniform estimate of the modulus of continuity.

 \item Our formula for the Hausdorff dimension of the graph is
proved via a lower $L^2$
increment estimate on large Cantor-type subsets whose $b$-adic orbits avoid the
discontinuities of the fractional part map.  This estimate then facilitates a
Frostman energy argument, which establishes the lower bound of the Hausdorff dimension; the upper bound follows directly from the moduli of continuity. 
\end{itemize}

Thus, all results in this paper rely essentially on suitable bounds for the covariance of the fractional Wiener--Weierstrass bridge. These bounds are delicate, due to the fractal nature of the covariance function (see Proposition~2.6 and Corollary~2.7 in \cite{SchiedZhang2024}).
The rest of the paper is organized as follows.  Section~\ref{sec:main-results}
states the main results.  Section~\ref{sec:proofs} gives the proofs.  
\section{Main results}\label{sec:main-results}

On a given probability space $(\Omega,\mathcal F,\Pp)$, let $W_H=\{W_H(t):0\le t\le1\}$ be a continuous  fractional
Brownian motion with Hurst parameter $H\in(0,1)$.
Furthermore, let $\kappa:[0,1]\to[0,1]$ be a deterministic function satisfying
$        \kappa(0)=0$, $\kappa(1)=1$,
and suppose that $\kappa$ is H\"older continuous with some exponent
$  \tau\in(H,1]$.
The corresponding fractional Brownian bridge is
\begin{equation*}
   B_H(u):=W_H(u)-\kappa(u)W_H(1),\qquad 0\le u\le1.
\end{equation*}
 For instance, with the choice
$\kappa(t):=\frac{1}{2}(1+t^{2H}-(1-t)^{2H})
$,
 the law of $B_H$ is equal to the law of $W_H$ conditioned on the event $\{W_H(1)=0\}$; see~\cite{GasbarraSottinenValkeila}. However, the specific form of $\kappa$ will not be needed in the sequel. Both $\kappa$ and $B_H$ will be fixed throughout this paper. We denote by $\{x\}$  the fractional part of $x\ge0$.
 
 \begin{definition}
\label{WW}   \cite{SchiedZhang2024} For $\alpha\in(0,1)$ and $b\in\{2,3,\dots\}$, the stochastic process
\begin{equation}\label{eq:ww}
Y(t):=
\sum_{n=0}^\infty \alpha^nB_H(\{b^n t\}),\qquad 0\le t\le 1,
\end{equation}
is called the \emph{fractional Wiener--Weierstrass  bridge} with parameters $\alpha$, $b$, and $H$.\footnote{The same process
is denoted by $X$ in \cite{SchiedZhang2024}; we use the notation of
\cite{SchiedZhang2026}.}
\end{definition}

Since $B_H$ is continuous and $B_H(0)=B_H(1)=0$, the series in
\eqref{eq:ww} converges uniformly for every sample path.   For a continuous function $f$ and $q>0$, write
\begin{equation*}
 V_n^{(q)}(f,t)=\sum_{k=0}^{\floor{tb^n}-1}|f((k+1)b^{-n})-f(kb^{-n})|^q,\qquad 0\le t\le1.
\end{equation*}
Theorem~2.3 of \cite{SchiedZhang2024} states that if 
\[H<K :=\min\{1,-\log_b\alpha\},
\]
 then  the $b$-adic $(1/H)^{\text{th}}$ variation of $Y$ is a deterministic linear function of $t$, namely,
\begin{equation}\label{eq:det-var}
V_n^{(1/H)}(Y,t)\longrightarrow\frac{2^{1/(2H)}\Gamma\big(\frac{H+1}{2H}\big)} {\sqrt\pi\,(1-\alpha^2b^{2H})^{1/(2H)}}\,t,     \qquad 0\le t\le1,\text{ $\Pp$-a.s.}
\end{equation}
If  $H>K$, then the same theorem 
 yields the existence of a
finite and strictly positive --- and possibly random --- coefficient $V$ such that $\Pp$-a.s.,
\begin{equation}\label{eq:V-def}
  V_n^{(1/K)}(Y,t)\longrightarrow Vt,  \qquad 0\le t\le1.
\end{equation}
The following theorem shows that the coefficient $V$ is a non-degenerate random variable, thereby resolving Open Problem~\ref{op:nondegeneracy}.

\begin{theorem}\label{thm:op1}
Assume $H>K$ and let $V$ be the limiting random
 coefficient in \eqref{eq:V-def}. Then $V\in L^q$ for every
$1\le q<\infty$, and the law of $V$ is absolutely continuous with respect
to Lebesgue measure.  In particular, $V$ is not almost surely constant.
\end{theorem}

As already noted in \cite[Remark 3.2]{SchiedZhang2024}, the random variable $V$ can be represented as a certain mixture of $(1/K)^{\text{th}}$ powers of the absolute values of certain Wiener integrals with integrator $W_H$. This representation will be the starting point for our proof of Theorem \ref{thm:op1}, which is given in Section \ref{subsec:proof-op1}. An illustration of the empirical distribution of $V$ can be found in the two histograms of \cite[Figure 1]{SchiedZhang2024}.

We next turn to Open Problem~\ref{op:critical-variation}. It concerns the $b$-adic power-variation
normalization  in the critical case $H=K$. While the correct normalization for the case $H=K=1/2$ was identified in \cite[Theorem 2.4]{SchiedZhang2024}, the general case remained open. The following theorem identifies   the correct $b$-adic power-variation
normalization for arbitrary $H\in(0,1)$.

\begin{theorem}[Critical rescaled power variation]\label{thm:critical-power}
Assume $H=K$.  Then, $\Pp$-a.s.~for every $t\in[0,1]$,
\begin{equation}\label{eq:critical-Phi-variation}
\lim_{n\uparrow\infty}
\frac1{n^{1/(2H)}}
\sum_{k=0}^{\floor{tb^n}-1}
\abs{Y((k+1)b^{-n})-Y(kb^{-n})}^{1/H}
=
\frac{2^{1/(2H)}\Gamma\big(\frac{H+1}{2H}\big)}{\sqrt\pi}\,t.
\end{equation}
In particular, almost all sample paths of $Y$ have infinite $b$-adic $(1/H)^{\text{th}}$ variation:
\begin{equation}
 V_n^{(1/H)}(Y,t)  \longrightarrow+\infty\qquad\text{for every $t\in(0,1]$ $\Pp$-a.s.}\label{eq:infinite-QV}
\end{equation}
\end{theorem}

Note that the constant factor on the right-hand side of \eqref{eq:critical-Phi-variation} can also be expressed as $\mathbb E[|Z|^{1/H}]$, where $Z$ is a standard normal random variable.

Comparing \eqref{eq:infinite-QV} with \eqref{eq:det-var} and \eqref{eq:V-def}, one sees that  the critical case $H=K$ differs from the  sub- and supercritical cases in that   power variation at $p:=1/H=1/K$ is infinite and hence insufficient to capture the actual fluctuations of the sample paths. It was therefore conjectured in \cite{SchiedZhang2024} that, in the critical case,  the power function $x^p$ should be amended with a logarithmic correction term of the form $(-\log x)^{-p/2}$ so as to render a finite and strictly positive limit. The following theorem addresses Open Problem \ref{op:critical-Phi-variation}  by proving that conjecture.

\begin{theorem}[Critical $\Phi$-variation]\label{thm:critical-phi}
Assume $H=K$.  Let $\Phi$ be a nonnegative continuous nondecreasing
function on $[0,\infty)$ such that $\Phi(0)=0$ and
\[  \Phi(x)=x^{1/H}(-\log x)^{-1/(2H)},\qquad 0<x<\frac1e.
\]
Then, $\Pp$-a.s.~for every $t\in[0,1]$,
\[
 \lim_{n\uparrow\infty}
 \sum_{k=0}^{\floor{tb^n}-1}
 \Phi\!\left(\abs{Y((k+1)b^{-n})-Y(kb^{-n})}\right)
 =
 \frac{2^{1/(2H)}\Gamma\big(\frac{H+1}{2H}\big)}
      {\sqrt\pi\,(-\log\alpha)^{1/(2H)}}\,t .
\]
\end{theorem}

Let us next state two corollaries of Theorems \ref{thm:critical-power} and \ref{thm:critical-phi}.  

\begin{corollary}\label{cor:critical-variation} For $H=K$, we have $\Pp$-a.s.~for all $t\in(0,1]$,
\begin{align*}
V_n^{(q)}(Y,t)\longrightarrow\begin{cases}+\infty&\text{for $q\le 1/H$,}\\
0&\text{for $q>1/H$.}
\end{cases}
\end{align*}
\end{corollary}

The preceding corollary admits the following interpretation.
Following \cite{HanSchiedHurst}, a continuous function $f:[0,1]\to\mathbb R$ admits the roughness exponent $R\in[0,1]$ if 
\begin{equation*}
V_n^{(q)}(f,1)\longrightarrow\begin{cases}+\infty&\text{for $q<1/R$,}\\
0&\text{for $q>1/R$.}
\end{cases}
\end{equation*}
It was shown in \cite{SchiedZhang2024} that the roughness exponent of the sample paths of $Y$ is almost surely equal to $H\wedge K$, provided that either $H\neq K$ or $H=K=1/2$. In conjunction with Corollary~\ref{cor:critical-variation},  this fact immediately yields the following result.

\begin{corollary}[Roughness exponent of $Y$]\label{cor:roughness-exponent} For all parameter values, the sample paths of $Y$ admit almost surely the roughness exponent $H\wedge K$.
\end{corollary}

Our final main result concerns Open Problem~\ref{op:hausdorff}.   Theorem~2.6 in \cite{SchiedZhang2026} states that  the Hausdorff
dimension of the graph of $Y$ is almost surely equal to $(2-H)\vee(2-K)$, provided that $K>2H-1$.  It was conjectured in \cite{SchiedZhang2026} that the same formula holds for all combinations of $H$ and $K$. The following theorem establishes the dimension formula in the regime where $H>K$ and $H>1/2$, which proves the conjecture; see the subsequent Corollary~\ref{cor:full}. This is achieved by establishing another conjecture from \cite{SchiedZhang2026}, namely that for $K\leq 2H-1$,  a lower bound of the form
\[\|Y(t)-Y(s)\|_{L^2}\ge \frac1 L |t-s|^K,
\]
which cannot hold uniformly in $s,t\in[0,1]$, instead holds on Cantor-like sets $T_N$, whose Hausdorff dimension approaches 1 as $N\uparrow\infty$; see 
Proposition~\ref{prop:lower-on-TN} and Lemma~\ref{lem:SN}.

\begin{theorem}[Hausdorff dimension in the high-Hurst regime]\label{thm:hausdorff-high-hurst}
If
$  H>K$,  $H>1/2$, then  
\begin{equation*}
  \dimH\Graph(Y)=2-K
  \qquad\text{$\Pp$-a.s.}
\end{equation*}
\end{theorem}

Combining Theorem~\ref{thm:hausdorff-high-hurst} with
\cite[Theorem~2.6]{SchiedZhang2026} yields the full formula:

\begin{corollary}[Hausdorff dimension of the graph of $Y$]\label{cor:full}
For all $H\in(0,1)$, $b\ge2$, and $\alpha\in(0,1)$,\begin{equation}
  \dimH\Graph(Y)=(2-H)\vee(2-K)
  \qquad\text{$\Pp$-a.s.}
  \label{eq:full-dimension-formula}
\end{equation}
\end{corollary}

\begin{proof}
If $H>K$ and $H>1/2$, \eqref{eq:full-dimension-formula} is Theorem~\ref{thm:hausdorff-high-hurst}.  If $H>K$ and
$H\le1/2$, then $2H-1\le0<K$, so $K>2H-1$; this case is covered by \cite[Theorem~2.6]{SchiedZhang2026}.  If
$H\le K$, then $K>2H-1$ as well, and the same cited theorem applies.  These
cases exhaust all possible relations between $H$ and $K$.
\end{proof}

\section{Proofs}\label{sec:proofs}

In the proofs of all of our main results, we need to deal with $b$-adic increments of $Y$, which involve the fractional part map $x\mapsto\{x\}$. To this end, we will use a trick from \cite{SchiedZZhang,SchiedZhang2024}, which represents fractional parts of $b$-adic rationals via external randomization. 
 Let $U_1,U_2,\ldots$ be
independent  random variables on a separate probability space $(\Omega_R,\mathcal F_R,\Pp_R)$ that have a uniform distribution on $\{0,\ldots,b-1\}$. Define
\begin{equation}
     R_m=\sum_{i=1}^m U_i b^{i-1},\qquad m\ge1.
     \label{eq:randomized-digits}
\end{equation}
Then $R_m$ has a uniform distribution on $\{0,1,\dots,b^m-1\}$, and so \eqref{eq:randomized-digits} can be seen as a randomization for $b$-adic digits. Its significance for dealing with fractional parts is owed to the fact that
\begin{equation*}
\{b^{-m}R_n\}=b^{-m}R_m\qquad\text{for $m\le n$.}
\end{equation*}
The expectation with respect to $\Pp_R$ will be denoted by $\E_R$.  

\subsection{Proof of Theorem~\ref{thm:op1}}\label{subsec:proof-op1}

We assume throughout this subsection that $H>K$ and put $p=1/K$.  Using the representation \eqref{eq:randomized-digits}, we define for a continuous
function $f:[0,1]\to\R$ with $f(0)=f(1)=0$,
\[
  \mathcal Z(f,R)
  :=
  \sum_{m=1}^\infty
  \alpha^{-m}
  \left(
  f\left(\frac{R_m+1}{b^m}\right)
  -f\left(\frac{R_m}{b^m}\right)
  \right),
\]
whenever the series is well defined; for $R_m=b^m-1$, the endpoint $(R_m+1)b^{-m}=1$ causes no ambiguity
because $f(1)=f(0)$. In the regime $H>K$, one may choose
$\gamma\in(K,H)$.  Then $\alpha b^\gamma=b^{\gamma-K}>1$, and the sample
paths of $B_H$ are $\gamma$-H\"older $\Pp$-almost surely.  Hence
\cite[Proposition~A.2]{SchiedZhang2024} applies pathwise to
$f(t)=B_H(t,\omega)$ and gives
\begin{equation}
  V=\E_R[\abs{\mathcal Z(B_H,R)}^p],
  \qquad p=1/K.
  \label{eq:V-branch-representation}
\end{equation}

\begin{lemma}\label{lem:op1-detbound}
Let $r\ge1$ and $\gamma>K$.  If $f(0)=f(1)=0$ and
\[
  M_\gamma(f):=
  \sup_{m\ge1}\sup_{0\le j<b^m}
  b^{\gamma m}\abs{f((j+1)b^{-m})-f(jb^{-m})}<\infty,
\]
then
\[
  \E_R[\abs{\mathcal Z(f,R)}^r]
  \le
  \left(\frac{M_\gamma(f)}{1-b^{K-\gamma}}\right)^r .
\]
\end{lemma}

\begin{proof}
Since $K<H$, we have $\alpha^{-m}=b^{Km}$. Hence, for every realization of
$R$,
\[
\begin{split}
\abs{\mathcal Z(f,R)}
&\le
\sum_{m=1}^\infty b^{Km}
\abs{f((R_m+1)b^{-m})-f(R_mb^{-m})}  \le
M_\gamma(f)\sum_{m=1}^\infty b^{(K-\gamma)m}
\le
\frac{M_\gamma(f)}{1-b^{K-\gamma}} .
\end{split}
\]
Taking the $r$-th power and the expectation over $R$ gives the result.
\end{proof}

The next estimate shows all moments of $M_\gamma(B_H)$ are finite.

\begin{lemma}\label{lem:op1-moment}
For $\gamma\in(K,H)$ and $r\in[1,\infty)$, we have
$
  \E[M_\gamma(B_H)^r]<\infty .
$
\end{lemma}

\begin{proof}
For integers $m\ge1$ and $0\le j<b^m$, put
$
  G_{m,j}:=b^{\gamma m}
  \big(B_H((j+1)b^{-m})-B_H(jb^{-m})\big)
$.
Since $B_H(t)=W_H(t)-\kappa(t)W_H(1)$,
\[
\begin{split}
\Var(G_{m,j})
&\le
2b^{2\gamma m}\Var(W_H((j+1)b^{-m})-W_H(jb^{-m})) 
+2b^{2\gamma m}
  \abs{\kappa((j+1)b^{-m})-\kappa(jb^{-m})}^2\Var(W_H(1)).
\end{split}
\]
The first variance is $b^{-2Hm}$, and the H\"older continuity of
$\kappa$ gives
$\abs{\kappa((j+1)b^{-m})-\kappa(jb^{-m})}\le Cb^{-\tau m}$, with
$\tau>H$.  Therefore, there are constants $C_0,C_1>0$ such that
$     \Var(G_{m,j})\le C_0 b^{-2(H-\gamma)m}$ and $  \Pp(|G_{m,j}|>x)
  \le 2e^{-C_1x^2b^{2(H-\gamma)m}}
$
for all $x\ge1$, all $m\ge1$, and all $0\le j<b^m$.  Hence
\begin{equation*}
\begin{split}
\Pp(M_\gamma(B_H)>x)
&\le
\sum_{m=1}^\infty\sum_{j=0}^{b^m-1}\Pp(|G_{m,j}|>x)  \le
2\sum_{m=1}^\infty
b^m e^{-C_1x^2b^{2(H-\gamma)m}}
\le C_2e^{-C_3x^2},
\end{split}
\end{equation*}
for suitable constants $C_2,C_3>0$. 
The tail integral formula for moments now
gives $\E[M_\gamma(B_H)^r]<\infty$.
\end{proof}

\begin{proof}[Proof of Theorem \ref{thm:op1}]
Choose $\gamma\in(K,H)$.  By \eqref{eq:V-branch-representation}, we have
$
  V=\E_R[\abs{\mathcal Z(B_H,R)}^p] 
$.
For every $q\ge1$, Jensen's inequality and Lemma
\ref{lem:op1-detbound} with exponent $r=pq$ give
\[
  V^q = \big(\E_R[\abs{\mathcal Z(B_H,R)}^p]\big)^q
  \le     \E_R[\abs{\mathcal Z(B_H,R)}^{pq}]     \le     \left(\frac{M_\gamma(B_H)}{1-b^{K-\gamma}}\right)^{pq}.
\]
Taking expectations and applying Lemma \ref{lem:op1-moment} yields
$V\in L^q$ for every $1\le q<\infty$.

It remains to prove absolute continuity of the law of $V$.  Let $\mathcal U=\{0,\ldots,b-1\}^{\mathbb N}$, and let $\mu_U$ be the
product measure on $\mathcal U$ whose marginals are uniform on
$\{0,\ldots,b-1\}$.  For $u=(u_i)_{i\ge1}\in\mathcal U$, write
\[     r_m(u)=\sum_{i=1}^m u_i b^{i-1}
\]
Then the law of $(r_1,r_2,\dots)$ under $\mu_U$ is equal to the law of  $(R_1,R_2,\dots)$ under $\Pp_R$. 
Now we define
\[     \mathcal Z_u     :=     \sum_{m=1}^\infty     b^{Km}     \left(     B_H\left(\frac{r_m(u)+1}{b^m}\right)     -     B_H\left(\frac{r_m(u)}{b^m}\right)     \right).
\]
The series converges in $L^2(\Omega)$, uniformly in $u\in\mathcal U$ in
the following sense.  Since
$     \|B_H(t)-B_H(s)\|_{L^2(\Omega)}\le C|t-s|^H
$,
\[
\begin{split}     \|\mathcal Z_u\|_{L^2(\Omega)}     &\le     \sum_{m=1}^\infty     b^{Km}     \left\|     B_H\left(\frac{r_m(u)+1}{b^m}\right)     -     B_H\left(\frac{r_m(u)}{b^m}\right)     \right\|_{L^2(\Omega)}                          \le     C\sum_{m=1}^\infty b^{(K-H)m}     <\infty.
\end{split}
\]
Using the
canonical Gaussian Hilbert space $\mathcal H_H$ introduced in
the Appendix, there is an element $h_u\in\mathcal H_H$ such that
$\mathcal Z_u=W_H(h_u)$, and
$        \sup_{u\in\mathcal U}\|h_u\|_{\mathcal H_H}<\infty 
$.
Moreover, since $H>K$, we may apply \cite[Lemma~3.1]{SchiedZhang2024},
which gives
\begin{equation}     \|h_u\|_{\mathcal H_H}^2     =     \E\big[\abs{\mathcal Z_u}^2\big]     >0     \qquad\text{for every }u\in\mathcal U .
\label{eq:norm-hu-strictly-positive}
\end{equation}
For each fixed $h\in\mathcal H_H$, the map
$u\mapsto \langle h_u,h\rangle_{\mathcal H_H}$ is measurable, since it is
obtained as the limit of the corresponding finite partial sums.

We next choose one Gaussian direction along which a positive-measure set of
the functionals $h_u$ has nonzero projection.
By \eqref{eq:norm-hu-strictly-positive}, 
there exists a unit vector $e\in\mathcal H_H$ such that  $\mu_U(\{u:a_u\ne0\})>0$, where
$a_u:=\langle h_u,e\rangle_{\mathcal H_H}.$  Put
\[     e^\perp:=\{h\in\mathcal H_H:\langle h,e\rangle_{\mathcal H_H}=0\},     \qquad     \mathcal G:=\sigma\big(W_H(h):h\in e^\perp\big).
\]
Then $W_H(e)$ is standard normal and independent of $\mathcal G$.  Also
write
$h_u^\perp:=h_u-a_ue\in e^\perp$.
Thus
$     W_H(h_u)=W_H(h_u^\perp)+a_uW_H(e)
$
and so
\[     V=\int_{\mathcal U}|W_H(h_u)|^p\,\mu_U(\mathrm{d}u)=\int_{\mathcal U}|W_H(h_u^\perp)+a_uW_H(e)|^p\,\mu_U(\mathrm{d}u).
\]
Let $Q$ be a regular conditional distribution of $V$ given $\mathcal G$.
Since $W_H(h_u^\perp)$ is $\mathcal G$-measurable for each $u$ and
$W_H(e)$ is independent of $\mathcal G$, we have, for $\Pp$-a.e.~$\omega$ and all Borel sets $A\subset\mathbb R$,
\begin{equation}     Q(\omega,A)     =\Pp(V\in A\mid\mathcal G)(\omega)     =\Pp(g_\omega(W_H(e))\in A),     \label{eq:V-cond-law}
\end{equation}
where
\[     g_\omega(x):=     \int_{\mathcal U}|W_H(h_u^\perp)(\omega)+a_ux|^p\,\mu_U(\mathrm{d}u).
\]
We will show below that for a.e.~$\omega$, the law of $g_\omega(W_H(e))$ is absolutely continuous with respect to the Lebesgue measure. That is, $Q(\omega,A)=0$ for a.e.~$\omega$ if $A$ is a Lebesgue null set. 
Once this has been achieved, the absolute continuity of the law of $V$ follows by taking expectations with respect to $\mathbb P$ in \eqref{eq:V-cond-law}. 

We now analyze the regularity of $g$, where we suppress its dependence on  $\omega$ to simplify the
notation.  Since
$     \sup_{u\in\mathcal U}|a_u|     \le     \sup_{u\in\mathcal U}\|h_u\|_{\mathcal H_H}     <\infty
$
and
\[
\begin{split}     \E\left[\int_{\mathcal U}\abs{W_H(h_u^\perp)}^p\,\mu_U(\mathrm{d}u)\right]     &=     \int_{\mathcal U}\E\big[\abs{W_H(h_u^\perp)}^p\big]\,\mu_U(\mathrm{d}u)    \le     c_p\sup_{u\in\mathcal U}\|h_u\|_{\mathcal H_H}^p     <\infty,
\end{split}
\]
where $c_p$ is the $p$-th absolute moment of a standard normal random
variable, we have
\begin{equation}     \int_{\mathcal U}|W_H(h_u^\perp)|^p\,\mu_U(\mathrm{d}u)<\infty     \qquad\text{$\Pp$-a.s.}     \label{eq:perp-integrability}
\end{equation}
On this full-measure event, $g(x)<\infty$ for every $x\in\mathbb R$.
Since $p>1$, the function $\psi$ defined by $\psi(0):=0$ and
$     \psi(z):=\operatorname{sgn}(z)|z|^{p-1}
$
for $z\ne0$ is continuous on $\mathbb R$.  Moreover, for each fixed $u$,
the derivative of $x\mapsto |W_H(h_u^\perp)+a_ux|^p$ is equal to 
$     pa_u\psi(W_H(h_u^\perp)+a_ux)
$.
Let $L:=\sup_{u\in\mathcal U}|a_u|<\infty$.  For every $M<\infty$,
there exists a constant $C_M$ such that, for $|x|\le M$,
\[     |a_u\psi(W_H(h_u^\perp)+a_ux)|     \le L(|W_H(h_u^\perp)|+LM)^{p-1}     \le C_M(1+|W_H(h_u^\perp)|^{p-1}).
\]
By \eqref{eq:perp-integrability}, the right-hand side is $\Pp$-a.s.~$\mu_U$-integrable.  Hence, dominated convergence, applied first to the difference
quotients and then to the  derivative of the integrand, shows that $g$ is
continuously differentiable and that
\[     g'(x)     =     p\int_{\mathcal U}     a_u\,\operatorname{sgn}(W_H(h_u^\perp)+a_ux)     |W_H(h_u^\perp)+a_ux|^{p-1}\,\mu_U(\mathrm{d}u).
\]
Next, the function $g$ is strictly convex.  Indeed, for every $u$ with
$a_u\ne0$, the function
$        x\longmapsto |W_H(h_u^\perp)+a_ux|^p
$
is strictly convex, because $p>1$.  Since $\mu_U(\{u:a_u\ne0\})>0$,
integration over $u$ preserves strict convexity.

If
$g:\mathbb R\to\mathbb R$ is $C^1$ and strictly convex, then for every
Lebesgue-null Borel set $A\subset\mathbb R$, we have 
$     \lambda(g^{-1}(A))=0
$,
where $\lambda$ denotes Lebesgue measure; see \cite[Lemma 7.25]{Rudin1987RealComplex}. 
Applying this fact $\omega$-wise to the function $g_\omega$, and using that the law of $W_H(e)$
has a density, we obtain for every Lebesgue-null Borel set $A\subset\mathbb R$
that
$     \Pp(g_\omega(W_H(e))\in A)=0
$
for $\Pp$-a.e.~$\omega$.  This concludes the proof.
\end{proof}

\subsection{Proof of Theorem~\ref{thm:critical-power}}\label{subsec:proof-critical-power}

Throughout this section, we work in the
critical case
$     H=K$, or equivalently $ \alpha=b^{-H}$.
 For $0\le k\le b^n-1$, set
\[I_{n,k}:=[kb^{-n},(k+1)b^{-n}],\qquad     \Delta Y(I_{n,k}):     =Y((k+1)b^{-n})-Y(kb^{-n}),\qquad     \xi_{n,k}:=\alpha^{-n}\Delta Y(I_{n,k}) .
\]
Similarly, we define $\Delta W_H(I_{n,k})$,  $\Delta B_H(I_{n,k})$, and $\Delta\kappa(I_{n,k})$. 
Substituting the series defining $Y$ into $\Delta Y(I_{n,k})$ gives
\[
\Delta Y(I_{n,k})
=
\sum_{m=0}^{\infty}\alpha^m
\left(
B_H\!\left(\left\{b^m(k+1)b^{-n}\right\}\right)
-
B_H\!\left(\left\{b^m k b^{-n}\right\}\right)
\right).
\]
If $m\ge n$, then
$b^m(k+1)b^{-n}-b^mkb^{-n}=b^{m-n}\in\Z$, so the two fractional parts are
equal and the $m$-th summand has zero increment.  Thus only the terms
$m=0,\ldots,n-1$ remain.  We reindex these terms by setting $r=n-m$.
Then $r=1,\ldots,n$, and
$     b^mkb^{-n}={k}{b^{-r}}
$.
If $k_r:=k\bmod b^r$, then
$     \left\{{k}{b^{-r}}\right\}={k_r}{b^{-r}}
$. 
Similarly,
$     \left\{(k+1){b^{-r}}\right\}     =
       (k_r+1){b^{-r}},
$
except in the wrap-around case $k_r=b^r-1$, where the fractional part is
$0$.  In \eqref{eq:branch-representation} below we write the corresponding
right endpoint as $1$; this causes no change because $B_H(1)=B_H(0)=0$.
Finally, since $H=K$, we have 
$     \alpha^{-n}\alpha^m=\alpha^{-r}=b^{Hr}
$.
Therefore
\begin{equation}
\begin{split}     \xi_{n,k}     =     \sum_{r=1}^n b^{Hr}     \Big(B_H((k_r+1)b^{-r})-B_H(k_r b^{-r})\Big)= \sum_{r=1}^n b^{Hr}
       \Delta B_H(I_{r,k_r}),     \qquad     k_r:=k\bmod b^r .
\end{split}     \label{eq:branch-representation}
\end{equation}

\subsubsection{Expectation asymptotics}

 Let $U_1,U_2,\ldots$ and $R_1,R_2,\ldots$ be the random  variables on the auxiliary probability space $(\Omega_R,\mathcal F_R,\Pp_R)$ as defined at the beginning of Section \ref{sec:proofs}. 
Then $R_r$ is uniform on $\{0,\ldots,b^r-1\}$ and
$R_r=R_s\bmod b^r$ for $r\le s$.  The map
$     k\longmapsto (k\bmod b,\ldots,k\bmod b^n)
$
pushes the uniform measure on $\{0,\ldots,b^n-1\}$ to the law of
$(R_1,\ldots,R_n)$.  Hence
\begin{equation}     b^{-n}\sum_{k=0}^{b^n-1}\E[\abs{\xi_{n,k}}^p]     =     \overline{\E}\bigg[\bigg|\sum_{r=1}^n b^{Hr}\Delta B_H(I_{r,R_r})\bigg|^p\bigg],    \label{eq:branch-expectation}
\end{equation}
where $\overline{\E}$ is the expectation with respect to the product measure $\overline{\Pp}:=\Pp\otimes\Pp_R$ on $\Omega\times\Omega_R$.
 For the remainder of this section, we set
\[     p=\frac1H,\qquad     c_H=\E[\abs{Z}^{1/H}]     =\frac{2^{1/(2H)}\Gamma\big(\frac{H+1}{2H}\big)}{\sqrt\pi},
\]
where $Z$ is standard normal. In the sequel $C$ will denote a generic constant that may change in every step. The following lemma identifies the asymptotic expected contribution of the normalized critical increments.

\begin{lemma}\label{lem:expectation}
As $n\to\infty$,
\begin{equation}
   n^{-p/2}b^{-n}\sum_{k=0}^{b^n-1}\E[\abs{\xi_{n,k}}^p]     \longrightarrow c_H .\label{eq:expectation-lemma-first-assertion}
\end{equation}
Moreover, for every $b$-adic interval $J\subset[0,1]$,
\begin{equation}     n^{-p/2}b^{-n}     \sum_{\{0\le k<b^n:\,[kb^{-n},(k+1)b^{-n}]\subset J\}}     \E[\abs{\xi_{n,k}}^p]     \longrightarrow c_H |J|,\label{eq:expectation-lemma-second-assertion}
\end{equation}
and the same convergence holds with $J=[0,t]$ for every
$t\in[0,1]$.
\end{lemma}

\begin{proof}
First, we remove the bridge correction.  Put
\[     G_n=\sum_{r=1}^n b^{Hr}\Delta W_H(I_{r,R_r}),\qquad     D_n=\sum_{r=1}^n b^{Hr}\Delta\kappa(I_{r,R_r}).
\]
Then
\[     \sum_{r=1}^n b^{Hr}\Delta B_H(I_{r,R_r})=G_n-D_nW_H(1).
\]
Since $\kappa$ is $\tau$-H\"older continuous and $\tau>H$,
\begin{equation}     \abs{D_n}\le C_0\sum_{r=1}^n b^{Hr}b^{-\tau r}\le C.       \label{eq:bridge-correction-bound}
\end{equation}
Thus $\norm{D_nW_H(1)}_{L^p(\Pp)}\le C_1$. 

 It remains to estimate the contribution of $G_n$.
Conditional on $R=(R_r)$, $G_n$ is centered Gaussian.  Let
$     \sigma_n^2(R)=\E[G_n^2]
$.
The diagonal contribution to $\sigma_n^2(R)$ is 
\[\sum_{r=1}^nb^{2Hr}\E[(\Delta W_H(I_{r,R_r}))^2]=\sum_{r=1}^nb^{2Hr}b^{-2Hr}=n.
\]
  Writing the off-diagonal
contribution as
\[     \eta_n     =     2\sum_{1\le r<s\le n}     b^{H(r+s)}     \Cov\big(\Delta W_H(I_{r,R_r}),\Delta W_H(I_{s,R_s})\big),
\]
we have $\sigma_n^2(R)=n+\eta_n$.  We now prove the estimate on $\eta_n$
needed for $\E_R[\sigma^p_n(R)]=\E_R[(n+\eta_n)^{p/2}]$. For $H\ge1/2$ we use nonnegativity and
an $L^1$ bound, while for $H<1/2$ we use an $L^{p/2}$ bound.  In the
estimates below, whenever a translated increment is evaluated outside
$[0,1]$, we use a two-sided fractional Brownian motion $W_H$ on $(\Omega,\mathcal F,\Pp)$ with covariance
$     \frac12(|u|^{2H}+|v|^{2H}-|u-v|^{2H})
$. This does not change any finite-dimensional distribution on $[0,1]$, and we  automatically have $W_H(0)=0$ $\Pp$-a.s.

Let us now consider the case $H\ge1/2$, in which the increments of fractional Brownian motion are non-negatively correlated.
For $r<s$, write $i=R_s$.  Since $R_r=i\bmod b^r$, stationarity of
increments gives
\begin{align}
&\E_R[\Cov\big(\Delta W_H(I_{r,R_r}),\Delta W_H(I_{s,R_s})\big)]\nonumber\\
&\quad =
b^{-s}\sum_{i=0}^{b^s-1}
\Cov\big(W_H(b^{-r}),     W_H((i+1)b^{-s}-\{ib^{-r}\})-W_H(ib^{-s}-\{ib^{-r}\})\big).\label{eq:cov-DeltaW}
\end{align}
We claim that 
the map $i\mapsto j_i:=b^s\{ib^{-r}\}-i
$ is one-to-one.  Indeed, if $j_i=j_{i'}$, then
$b^s$ divides $(i-i')(b^{s-r}-1)$; since $b^{s-r}-1$ is coprime to
$b$, $b^s$ divides $i-i'$, and hence $i=i'$.  Also
$-b^s<j_i<b^s$.  Since the summands in \eqref{eq:cov-DeltaW} are non-negative for $H\ge1/2$, we
may enlarge the sum to all integers $-b^s,\ldots,b^s-1$, and obtain
\begin{equation}
\begin{split}
&\E_R[\Cov\big(\Delta W_H(I_{r,R_r}),\Delta W_H(I_{s,R_s})\big)]        \\
&\quad\le
 b^{-s}\sum_{j=-b^s}^{b^s-1}
 \Cov\big(W_H(b^{-r}),W_H((j+1)b^{-s})-W_H(jb^{-s})\big)                 \\
&\quad=
 b^{-s}\Cov\big(W_H(b^{-r}),W_H(1)-W_H(-1)\big)
 =\frac{ b^{-s}}{2}\big(|b^{-r}+1|^{2H}-|1-b^{-r}|^{2H}\big)\le Cb^{-r-s}.
\end{split}\label{eq:cov-summation-inequality}
\end{equation}
  Therefore, in the critical case,
\begin{equation}     \E_R[\eta_n]     \le C\sum_{1\le r<s\le n} b^{H(r+s)}b^{-r-s}     \le C.                                                     \label{eq:eta-mean-bound}
\end{equation}

If $H<1/2$, then $p/2>1$.  We claim that, for $r<s$,
\begin{equation}
 \left\|
 \Cov\big(\Delta W_H(I_{r,R_r}),\Delta W_H(I_{s,R_s})\big)
 \right\|_{L^{p/2}(\Pp_R)}
 \le C\big((1+s)^{2/p}b^{-4Hs}+b^{-s}\big).                       \label{eq:covariance-lp-claim}
\end{equation}
To this end, we first argue as in \eqref{eq:cov-summation-inequality} to obtain
\begin{equation}
\begin{split}
&\E_R[\abs{\Cov(\Delta W_H(I_{r,R_r}),\Delta W_H(I_{s,R_s}))}^{p/2}]   \\& \le
 b^{-s}\sum_{j=-b^s}^{b^s-1}
 \abs{\Cov(W_H(b^{-r}),W_H((j+1)b^{-s})-W_H(jb^{-s}))}^{p/2}.
\end{split}\label{eq:ER-cov-bound-smallH}
\end{equation}
The covariance inside the absolute value is bounded by
\[
 C\left(
 \abs{\abs{(j+1)b^{-s}}^{2H}-\abs{jb^{-s}}^{2H}}
 +
 \abs{\abs{(j+1)b^{-s}-b^{-r}}^{2H}-\abs{jb^{-s}-b^{-r}}^{2H}}
 \right).
\]
Since $p/2>1$, we may use the Jensen-type inequality  $(u+v)^{p/2}\le 2^{p/2-1}(u^{p/2}+v^{p/2})$ for
$u,v\ge0$.
Let $N=b^s$, $h=b^{-s}$, and
$\delta_m=((m+1)h)^{2H}-(mh)^{2H}$, $m\ge0$.  Since $x\mapsto |x|^{2H}$
is even and monotone on each of the two half-lines,
\[
 \sum_{j=-N}^{N-1}\abs{|(j+1)h|^{2H}-|jh|^{2H}}^{p/2}
 \le 2\sum_{m=0}^N \delta_m^{p/2}.
\]
For the second difference, put $M=b^{s-r}$.  Then
\[
\begin{split}
\sum_{j=-N}^{N-1}
\abs{|(j+1-M)h|^{2H}-|(j-M)h|^{2H}}^{p/2}  \le
2\sum_{m=0}^{2N}\delta_m^{p/2},
\end{split}
\]
because the shifted indices $j-M$ lie between $-2N$ and $N$.  Therefore
the right-hand side in \eqref{eq:ER-cov-bound-smallH}  is at most
\[
 Cb^{-s}\sum_{j=0}^{2b^s}
 \big((j+1)b^{-s})^{2H}-(jb^{-s})^{2H}\big)^{p/2}.
\]
The terms $j=0,1$ contribute $O(b^{-2s})$.  For $j\ge2$, the mean-value
theorem gives
\[     ((j+1)b^{-s})^{2H}-(jb^{-s})^{2H}     \le C j^{2H-1}b^{-2Hs}.
\]
Since $p=1/H$, the total contribution of these terms  becomes (recalling our convention that the constant $C$ may change in every step),
\[     Cb^{-s}\sum_{j=2}^{2b^s}j^{1-p/2}b^{-s}     \le Cb^{-2s}\int_1^{2b^s}x^{1-p/2} \dd x\le C\big((1+s)b^{-2s}+b^{-ps/2}\big),
\]
which implies \eqref{eq:covariance-lp-claim}.
Hence, by Minkowski's inequality,
\begin{equation}
\begin{split}     \norm{\eta_n}_{L^{p/2}(\Pp_R)}     &\le     C\sum_{1\le r<s\le n}     b^{H(r+s)}\big((1+s)^{2/p}b^{-4Hs}+b^{-s}\big)                    \\     &\le     C\sum_{s=1}^\infty     \big(b^{(2H-1)s}+(1+s)^{2/p}b^{-2Hs}\big)     \le C .
\end{split}
\label{eq:eta-lp-bound}
\end{equation}

Let $\gamma=p/2$.  If $H\ge1/2$, then $\gamma\le1$, $\eta_n\ge0$,
and \eqref{eq:eta-mean-bound} implies
\[     0\le     \E_R\left[\left(1+\frac{\eta_n}{n}\right)^\gamma-1\right]     \le \E_R\left[\left(\frac{\eta_n}{n}\right)^\gamma\right]     \le \left(\frac{\E_R[\eta_n]}{n}\right)^\gamma     \longrightarrow0 .
\]
If $H<1/2$, then $\gamma>1$, and $n+\eta_n=\sigma_n^2(R)\ge0$.  The elementary
inequality
$     \abs{x^\gamma-y^\gamma}     \le C_\gamma\big(y^{\gamma-1}\abs{x-y}+\abs{x-y}^\gamma\big)
$, 
applied with $x=n+\eta_n\ge0$ and $y=n\ge0$, together with
\eqref{eq:eta-lp-bound}, gives
\[     \E_R[\abs{(n+\eta_n)^\gamma-n^\gamma}]     \le C(n^{\gamma-1}\E_R[\abs{\eta_n}]+\E_R[\abs{\eta_n}^\gamma])     =O(n^{\gamma-1})=o(n^\gamma).
\]
Thus in both cases,
\begin{equation}     \E_R[(n+\eta_n)^\gamma] = n^\gamma(1+o(1)).                  \label{eq:variance-asymptotic}
\end{equation}
Since for given $R$,  we have  $ \E[G_n^2]=\sigma_n^2(R)=n+\eta_n$ and in turn $\E[\abs{G_n}^p]=c_H (n+\eta_n)^{p/2}$,
\eqref{eq:variance-asymptotic} gives
\begin{equation}
       \overline{\E}[\abs{G_n}^p]=c_H n^{p/2}(1+o(1)).                          \label{eq:gaussian-moment-asymptotic}
\end{equation}
Combining \eqref{eq:bridge-correction-bound} and
\eqref{eq:gaussian-moment-asymptotic} with Minkowski's inequality gives
\begin{equation}
 \overline{\E}[\abs{G_n-D_nW_H(1)}^p]=c_H n^{p/2}(1+o(1)).
 \label{eq:lem-expectation-first-assertion-aux-conv}
\end{equation}
Indeed, while $\norm{D_nW_H(1)}_{L^p(\overline{\Pp})}=O(1)$, \eqref{eq:gaussian-moment-asymptotic} says
$
\norm{G_n}_{L^p(\overline{\Pp})}=c_H^{1/p}n^{1/2}(1+o(1))$. Hence
\[
\norm{G_n-D_nW_H(1)}_{L^p(\overline{\Pp})}=c_H^{1/p}n^{1/2}+O(1),
\]
 and raising to the
$p$-th power yields \eqref{eq:lem-expectation-first-assertion-aux-conv}.
Together with \eqref{eq:branch-expectation}, this proves \eqref{eq:expectation-lemma-first-assertion}, our  first assertion.

To prove the second assertion, let $J=[ab^{-m},(a+1)b^{-m}]$.  For $n>m$, write
$k=ab^{n-m}+l$, $0\le l<b^{n-m}$.  Since
$(ab^{n-m}+l)\bmod b^r=l\bmod b^r$ for $r\le n-m$, the sum of the first $n-m$
terms in \eqref{eq:branch-representation} is exactly $\xi_{n-m,l}$.
Denote the sum of the remaining $m$ terms by $\Upsilon_{n,l}$.  Each summand in $\Upsilon_{n,l}$ is a normalized bridge increment
over a $b$-adic interval, and hence has $L^p$-norm bounded by a constant
depending only on $H,p,\kappa$; therefore,
$     \sup_{n,l}\norm{\Upsilon_{n,l}}_{L^p}\le Cm 
$.
Using
$
      \big|\abs{x+y}^p-\abs{x}^p\big|     \le C_p\big(\abs{x}^{p-1}\abs{y}+\abs{y}^p\big)
$,
H\"older's inequality, and Jensen's inequality for the summation over the index $\ell$,
\[
\begin{split}
&n^{-p/2}b^{-n}
 \sum_{\ell=0}^{b^{n-m}-1}
\big|\E[\abs{\xi_{n,ab^{n-m}+\ell}}^p]
       -\E[\abs{\xi_{n-m,\ell}}^p]\big|                                     \\
&\qquad\le
 C n^{-p/2}b^{-m}
 \left(
 m\left(b^{-(n-m)}\sum_{\ell=0}^{b^{n-m}-1}     \E[\abs{\xi_{n-m,\ell}}^p]\right)^{(p-1)/p}
 +m^p\right)                                                       \\
&\qquad\le
 C n^{-p/2}b^{-m}
 \left((n-m)^{(p-1)/2}m+m^p\right)\longrightarrow0,
\end{split}
\]
where the last inequality uses \eqref{eq:expectation-lemma-first-assertion} with
$n-m$ in place of $n$.  Hence
\[
\begin{split}
n^{-p/2}b^{-n}
\sum_{\ell=0}^{b^{n-m}-1}\E[\abs{\xi_{n-m,\ell}}^p]                          =
b^{-m}\left(\frac{n-m}{n}\right)^{p/2}
(n-m)^{-p/2}b^{-(n-m)}
\sum_{\ell=0}^{b^{n-m}-1}\E[\abs{\xi_{n-m,\ell}}^p]
\longrightarrow c_Hb^{-m}.
\end{split}
\]
Thus the limit of the left-hand side in \eqref{eq:expectation-lemma-second-assertion} as $n\uparrow\infty$ is $c_H b^{-m}=c_H|J|$.  For finite unions of
$b$-adic intervals, \eqref{eq:expectation-lemma-second-assertion} follows by additivity.  For general $t$, let
\[     E_n(A):=n^{-p/2}b^{-n}     \sum_{\{0\le k<b^n:\,[kb^{-n},(k+1)b^{-n}]\subset A\}}     \E[\abs{\xi_{n,k}}^p] .
\]
If $s< u$ are $b$-adic with $t\in(s,u)$, then
$     E_n([0,s])\le E_n([0,t])\le E_n([0,u])
$, and a sandwich argument gives
$E_n([0,t])\to c_Ht$.
\end{proof}

\subsubsection{An averaged Frobenius estimate}

The following averaged Frobenius estimate is the main covariance input for the concentration argument in Section~\ref{subsection:concentration}.

\begin{lemma}\label{lem:frobenius-covariance}
Let
\[     A^{(n)}_{ij}:=\Cov(\xi_{n,i},\xi_{n,j}),     \qquad 0\le i,j\le b^n-1.
\]
Then
\begin{equation}     b^{-2n}\sum_{i,j=0}^{b^n-1}\abs{A^{(n)}_{ij}}^2\le C     \qquad\text{for all }n.                                    \label{eq:frobenius-bound}
\end{equation}
Also,
\begin{equation}     \Var(\xi_{n,k})\le Cn     \qquad\text{for all }n\text{ and }0\le k\le b^n-1.          \label{eq:variance-bound}
\end{equation}
\end{lemma}

\begin{proof}
First, we consider only the fractional Brownian part so as to deal with the bridge correction $\kappa$ later.  For $r\le s$, let
$I_r$ and $J_s$ be independent uniformly chosen $b$-adic intervals of
levels $r$ and $s$.  Put $q=b^{-(s-r)}$, $N=b^r$, and
\[     \gamma_{r,s}=b^{H(r+s)}\Cov(\Delta W_H(I_r),\Delta W_H(J_s)).
\]
Let $P_s$ be the level-$r$ parent of $J_s$.  The event that $P_s$ is
equal to $I_r$ or adjacent to it has probability at most $3/N$.  On this
event, \eqref{eq:two-scale-bound} gives $\abs{\gamma_{r,s}}\le Cq^\theta$, where $\theta:=H\wedge(1-H)$.  On the complementary
event, if there are  $h\ge1$ level-$r$ intervals between $P_s$  and $I_r$,
then the distance between $J_s$ and $I_r$ is at least $hb^{-r}$, and
\eqref{eq:separated-bound} gives
$\abs{\gamma_{r,s}}\le Cq^{1-H}h^{2H-2}$.  Since the probability of each value
of $h$ is at most $2/N$,
\begin{equation}
 \E[\abs{\gamma_{r,s}}^2]
 \le
 C\left(
 b^{-r}b^{-2\theta(s-r)}
 +
 b^{-2(1-H)(s-r)}b^{-r}
 \sum_{h=2}^{b^r}h^{4H-4}
 \right).                                                        \label{eq:covariance-average-square}
\end{equation}
The square roots of the right-hand side are summable over $1\le r\le s$.
Indeed, the square root of the first term in
\eqref{eq:covariance-average-square} sums to
\[     \sum_{r=1}^\infty\sum_{\delta=0}^\infty     b^{-r/2}b^{-\theta\delta}<\infty,\qquad \delta=s-r.
\]
For the second term,
\begin{equation*}
 b^{-r}\sum_{h=2}^{b^r}h^{4H-4}
 \le C\begin{cases}
 b^{-r},&H<3/4,\\
 rb^{-r},&H=3/4,\\
 b^{-4(1-H)r},&H>3/4.
 \end{cases}                                                     
\end{equation*}
Hence, the square root of the second term in \eqref{eq:covariance-average-square} is bounded by
\[
C b^{-(1-H)(s-r)}
\begin{cases}
b^{-r/2},&H<3/4,\\
r^{1/2}b^{-r/2},&H=3/4,\\
b^{-2(1-H)r},&H>3/4.
\end{cases}
\]
This is summable over $r\ge1$ and $s-r\ge0$.  Therefore
\begin{equation}     \sum_{1\le r\le s<\infty}\big(\E[\abs{\gamma_{r,s}}^2]\big)^{1/2}<\infty.                                                                      \label{eq:covariance-summability}
\end{equation}

 We now include the bridge correction terms.  Let
$     d_{r,j}=b^{Hr}\Delta\kappa(I_{r,j})$ and $     a_{r,j}=b^{Hr}\Cov(\Delta W_H(I_{r,j}),W_H(1))
$.
The H\"older regularity of $\kappa$ gives
\begin{equation}     \sup_j\abs{d_{r,j}}\le Cb^{-(\tau-H)r}.                    \label{eq:d-bound}
\end{equation}
Moreover, the covariance estimate \eqref{eq:two-scale-bound} yields
\begin{equation}     \sup_j\abs{a_{r,j}}\le Cb^{-\theta r}.                     \label{eq:a-bound}
\end{equation}
For intervals $I_{r,j}$ and $I_{s,k}$, the difference between the normalized
covariance for $B_H$-increments and the corresponding normalized covariance for
$W_H$-increments is exactly
$     -a_{r,j}d_{s,k}-d_{r,j}a_{s,k}+d_{r,j}d_{s,k}
$.
It is therefore bounded by $C \nu_r \nu_s$, where
$     \nu_r=b^{-(\tau-H)r}+b^{-\theta r}
$
and $\sum_r \nu_r<\infty$.  Combining this with
\eqref{eq:covariance-summability}, and using symmetry for
the terms $r>s$, we obtain
\begin{equation*}
 \sum_{r,s=1}^\infty
 \left(\frac1{b^rb^s}\sum_{i=0}^{b^r-1}\sum_{j=0}^{b^s-1}
 \abs{
 b^{H(r+s)}
 \Cov(\Delta B_H(I_{r,i}),\Delta B_H(I_{s,j}))
 }^2
 \right)^{1/2}
 <\infty.         
\end{equation*}
Equivalently, for $n\ge r\vee s$, one can average
over $i,j=0,\dots, b^n-1$ while replacing $I_{r,i}$ and $I_{s,j}$ by $I_{r,i\bmod b^r}$ and $I_{s,j\bmod b^s}$, respectively. Thus, noting that
\[
\begin{split}
A^{(n)}_{ij}
&=\sum_{r,s=1}^n
b^{H(r+s)}
\Cov(\Delta B_H(I_{r,i\bmod b^r}),\Delta B_H(I_{s,j\bmod b^s})),
\end{split}
\]
 Minkowski's inequality in the finite Hilbert space
$\ell^2(\{0,\ldots,b^n-1\}^2)$, endowed with the normalized counting measure on $\{0,\ldots,b^n-1\}^2$,  gives
\[
\begin{split}
 \left(
 b^{-2n}\sum_{i,j=0}^{b^n-1}\abs{A^{(n)}_{ij}}^2
 \right)^{1/2}
 &\le
 \sum_{r,s=1}^n
 \left(
\frac1{b^rb^s}\sum_{i=0}^{b^r-1}\sum_{j=0}^{b^s-1}\abs{
 b^{H(r+s)}
 \Cov(\Delta B_H(I_{r,i\bmod b^r}),\Delta B_H(I_{s,j\bmod b^s}))
 }^2
 \right)^{1/2}                                                \le C.
\end{split}
\]
This proves \eqref{eq:frobenius-bound}.

For \eqref{eq:variance-bound}, we use that \eqref{eq:two-scale-bound},
\eqref{eq:d-bound}, and \eqref{eq:a-bound} imply the pointwise bound
\[
 \abs{
 b^{H(r+s)}
 \Cov(\Delta B_H(I_{r,i\bmod b^r}),\Delta B_H(I_{s,i\bmod b^s}))
 }
 \le Cb^{-\theta\abs{r-s}}+C\nu_r\nu_s .
\]
Summing this over $1\le r,s\le n$ gives
$     \Var(\xi_{n,i})     \le C\sum_{r,s=1}^n b^{-\theta\abs{r-s}}+C     \le Cn
$.
This proves \eqref{eq:variance-bound}.
\end{proof}

\subsubsection{Concentration of the critical power variation}\label{subsection:concentration}

The following elementary Hermite expansion estimate turns Gaussian covariance bounds into covariance bounds for powers of absolute values.

\begin{lemma}\label{lem:hermite}
For every $p>1$ there is $C_p<\infty$ such that, for all centered
jointly Gaussian $U,V$,
\begin{equation}
\abs{\Cov(\abs{U}^p,\abs{V}^p)}
\le C_p
\begin{cases}
\abs{\Cov(U,V)}^p,&\text{for $1<p\le2$,}\\
\Var(U)^{p/2-1}\Var(V)^{p/2-1}\Cov(U,V)^2,&\text{for $ p\ge2$.}
\end{cases}                                                       \label{eq:gaussian-power-cov}
\end{equation}
\end{lemma}

\begin{proof}
If either variance is zero, both sides of \eqref{eq:gaussian-power-cov} are
zero, so assume the variances are positive.
Assume first that $U,V$ have unit variance and correlation $\rho$.
The function $g(x)=\abs{x}^p-\E[\abs{Z}^p]$ has Hermite rank $2$, because it
is even and, for a standard normal random variable $Z$,
\[     \E\big[g(Z)(Z^2-1)\big]     =\E[\abs{Z}^{p+2}]-\E[\abs{Z}^p]     =p\,\E[\abs{Z}^p]\ne0.
\]
By \cite[Lemma 1]{Arcones1994}, for some $C_p>0$,
\begin{equation}     \abs{\Cov(\abs{U}^p,\abs{V}^p)}\le C_p\rho^2.     \label{eq:unit-variance-estimate}              
\end{equation}
If $1<p\le2$, then $\abs{\rho}^2\le\abs{\rho}^p$; if $p\ge2$, we keep
$\rho^2$.

Now let $U,V$ have standard deviations $\sigma_U,\sigma_V>0$, and set
$\widetilde U=U/\sigma_U$, $\widetilde V=V/\sigma_V$.  Their correlation
is
$     \rho=\Cov(U,V)/(\sigma_U\sigma_V)
$.
For $1<p\le2$, the unit-variance estimate \eqref{eq:unit-variance-estimate}    with $|\rho|^p$ gives
\[
\begin{aligned}
\abs{\Cov(\abs{U}^p,\abs{V}^p)}
&=\sigma_U^p\sigma_V^p
  \abs{\Cov(|\widetilde U|^p,|\widetilde V|^p)}   \le C_p\sigma_U^p\sigma_V^p|\rho|^p
 = C_p|\Cov(U,V)|^p .
\end{aligned}
\]
This is why no separate variance factors appear in the case $1<p\le2$.
For $p\ge2$, the unit-variance estimate with $\rho^2$ gives
\[
\begin{aligned}
\abs{\Cov(\abs{U}^p,\abs{V}^p)}
&\le C_p\sigma_U^p\sigma_V^p\rho^2                         =C_p\sigma_U^{p-2}\sigma_V^{p-2}\Cov(U,V)^2           =C_p\Var(U)^{p/2-1}\Var(V)^{p/2-1}\Cov(U,V)^2 .
\end{aligned}
\]
This proves \eqref{eq:gaussian-power-cov}.
\end{proof}

\begin{lemma}\label{lem:power-law}
Almost surely, for every $t\in[0,1]$,
\[      S_n(t):=n^{-p/2}b^{-n}     \sum_{k=0}^{\floor{tb^n}-1}\abs{\xi_{n,k}}^p\longrightarrow c_Ht .                              
\]
\end{lemma}

\begin{proof}
If $1<p\le2$, then by Jensen's inequality and Lemmas
\ref{lem:hermite} and \ref{lem:frobenius-covariance},
\begin{equation*}
\begin{split}     \Var(S_n(t))     &\le C n^{-p}b^{-2n}     \sum_{0\le i,j<\floor{tb^n}}|A^{(n)}_{ij}|^p    \le C n^{-p}     \left(b^{-2n}\sum_{i,j=0}^{b^n-1}|A^{(n)}_{ij}|^2     \right)^{p/2}     \le Cn^{-p}.
\end{split}
\end{equation*}
Since $p>1$, this is summable in $n$.  If $p\ge2$, then Lemmas
\ref{lem:hermite} and \ref{lem:frobenius-covariance} give
\begin{equation*}
\begin{split}     \Var(S_n(t))     &\le C n^{-p}b^{-2n}     \sum_{0\le i,j<\floor{tb^n}} n^{p-2}|A^{(n)}_{ij}|^2   \le Cn^{-2}.
\end{split}
\end{equation*}
This is again summable.  Chebyshev's inequality and the Borel--Cantelli lemma thus imply that 
$     S_n(t)-\E[S_n(t)]\to0
$
$\Pp$-a.s.~for each fixed $t$.

By Lemma \ref{lem:expectation}, $\E[S_n(t)]\to c_Ht$ for every $t$; in
particular, this holds at every $b$-adic $t$.  Intersecting the
corresponding probability-one events over the
countable set of $b$-adic $t$'s gives convergence at all $b$-adic
points.  Since $t\mapsto S_n(t)$ is nondecreasing and $t\mapsto c_Ht$ is
continuous, a standard sandwich argument yields almost-sure convergence for every $t\in[0,1]$.\end{proof}

\begin{proof}[Proof of Theorem \ref{thm:critical-power}]
Since $\Delta Y(I_{n,k})=\alpha^n\xi_{n,k}$, $p=1/H$, and
$\alpha^p=b^{-1}$, we have
\[
\frac1{n^{p/2}}
\sum_{k=0}^{\floor{tb^n}-1}\abs{\Delta Y(I_{n,k})}^p
=
n^{-p/2}b^{-n}
\sum_{k=0}^{\floor{tb^n}-1}\abs{\xi_{n,k}}^p
=S_n(t).
\]
Lemma \ref{lem:power-law} gives the asserted limit with
$c_H=\E[\abs{Z}^{1/H}]$.
\end{proof}

\subsection{Proof of Theorem~\ref{thm:critical-phi} and Corollary~\ref{cor:critical-variation}}\label{subsec:proof-critical-phi}

The next modulus estimate ensures that the logarithmic gauge can be expanded uniformly along the $b$-adic partition.

\begin{lemma}\label{lem:modulus}
For $\Pp$-a.e.~$\omega$,
\begin{equation*}     \max_{0\le k<b^n}\abs{\xi_{n,k}(\omega)}\le C(\omega)n     \qquad\text{for all sufficiently large }n.                 
\end{equation*}
\end{lemma}
\begin{proof}
By \cite[Theorem~2.4(ii)]{SchiedZhang2026}, we have $\Pp$-a.s.,
$     |\Delta_{n,k}Y|\le C(\omega)b^{-Hn}n
$
uniformly in $k$, for all large $n$.  Since $H=K$,
$\alpha=b^{-H}$, and $\Delta_{n,k}Y=\alpha^n\xi_{n,k}$, the claim follows.
\end{proof}

\begin{proof}[Proof of Theorem \ref{thm:critical-phi}]
By \eqref{eq:branch-representation}, we have 
$     \Delta Y(I_{n,k})=\alpha^n\xi_{n,k}=b^{-Hn}\xi_{n,k}
$.
Put $L_\alpha=-\log\alpha=H\log b$.
For all sufficiently large $n$, Lemma \ref{lem:modulus} implies
$\abs{\Delta Y(I_{n,k})}<e^{-1}$ uniformly in $k$, and
\begin{equation}
\begin{split}
\Phi(\abs{\Delta Y(I_{n,k})})
&=
b^{-n}\abs{\xi_{n,k}}^p
\big(nL_\alpha-\log\abs{\xi_{n,k}}\big)^{-p/2},
\end{split}                                                        \label{eq:phi-exact}
\end{equation}
with the convention that the term is $0$ for $\xi_{n,k}=0$.
We also interpret $x^p(1+|\log x|)$ as $0$ at $x=0$.

By Lemma \ref{lem:modulus}, for all large $n$ and all
$k$ with $\xi_{n,k}\ne0$, $\log\abs{\xi_{n,k}}\le nL_\alpha/2$.  Hence the
mean-value theorem gives, for $x>0$ with $\log x\le nL_\alpha/2$,
\begin{equation}     \abs{(nL_\alpha-\log x)^{-p/2}-(nL_\alpha)^{-p/2}}     \le C n^{-p/2-1}\abs{\log x}.      \label{eq:mean-value-estimate-in-proof-critical-phi-thm}                      
\end{equation}
Also $x^p\abs{\log x}\le C_p$ on $0<x\le1$, while on
$1\le x\le Cn^{3/2}$ we have $\abs{\log x}\le C'\log n$.
In the sequel,  $O_\omega(\cdot)$ denotes a bound with a finite random constant.
Since Lemma \ref{lem:power-law} with $t=1$ gives
\[     b^{-n}\sum_{k=0}^{b^n-1}\abs{\xi_{n,k}}^p     =O_\omega(n^{p/2}),
\]
we obtain
\begin{equation}     b^{-n}\sum_{k=0}^{b^n-1}     \abs{\xi_{n,k}}^p(1+\abs{\log\abs{\xi_{n,k}}})     =O_\omega(n^{p/2}\log n).                                  \label{eq:log-moment-bound}
\end{equation}
Combining \eqref{eq:phi-exact} with the mean-value estimate \eqref{eq:mean-value-estimate-in-proof-critical-phi-thm}, for all large
$n$ and every $k$,
\begin{equation*}
\begin{split}
\left|
\Phi(\abs{\Delta Y(I_{n,k})})
-
L_\alpha^{-p/2} n^{-p/2}b^{-n}\abs{\xi_{n,k}}^p
\right|                                                        \le
C n^{-p/2-1}b^{-n}
\abs{\xi_{n,k}}^p(1+\abs{\log\abs{\xi_{n,k}}}).
\end{split}
\end{equation*}
Summing this pointwise bound over $0\le k<\floor{tb^n}$ and then bounding
the partial sum by the full sum gives
\begin{equation*}
\begin{split}
&\sup_{0\le t\le1}\Bigg|
\sum_{k=0}^{\floor{tb^n}-1}\Phi(\abs{\Delta Y(I_{n,k})})
-
L_\alpha^{-p/2}S_n(t)
\Bigg|              \le
C n^{-p/2-1}
b^{-n}\sum_{k=0}^{b^n-1}
\abs{\xi_{n,k}}^p(1+\abs{\log\abs{\xi_{n,k}}})
\le C(\omega)n^{-1}\log n\longrightarrow0 ,
\end{split}                                                        
\end{equation*}
where the last inequality is \eqref{eq:log-moment-bound}.
By Lemma \ref{lem:power-law}, we have
$     L_\alpha^{-p/2}S_n(t)\longrightarrow L_\alpha^{-p/2}c_Ht
$
$\Pp$-a.s.~for every $t\in[0,1]$.  Since $L_\alpha=-\log\alpha$, this is
exactly the claimed constant.

Finally, if one uses the convention $k=0,\ldots,\floor{tb^n}$, one adds at
most one term.  If $t<1$, this term is covered directly by Lemma
\ref{lem:modulus}; if $t=1$, the periodic representation of $Y$ makes the
increment from $1$ to $1+b^{-n}$ identical to the increment from $0$ to
$b^{-n}$.  Thus the added term is bounded by
\[     C(\omega) b^{-n} n^{3p/2}(nL_\alpha)^{-p/2}     =C(\omega)b^{-n}n^p\longrightarrow0,
\]
so the convention does not change the limit.
\end{proof}

\begin{proof}[Proof of Corollary~\ref{cor:critical-variation}]
The case $q= 1/H$ is trivial by Theorem \ref{thm:critical-power}. We take $q>0$ with $q\neq p:=1/H$ and let 
$\Psi(x):={x^q}/{\Phi(x)}$, which for $0<x<1/e$ is equal to
$x^{q-p}(-\log x)^{p/2}$. If $q>p$, then
\begin{align*}
V_n^{(q)}(Y,1)&\le \bigg(\max_{k=0,\dots,b^n-1} \Psi\!\left(\abs{Y((k+1)b^{-n})-Y(kb^{-n})}\right)\bigg)\sum_{k=0}^{b^n-1}
 \Phi\!\left(\abs{Y((k+1)b^{-n})-Y(kb^{-n})}\right).
\end{align*}
$\Pp$-a.s., the sum on the right converges to a finite limit by  Theorem \ref{thm:critical-phi}, while the maximum tends to zero, due to the uniform continuity of the sample paths of $Y$ and the fact that $\lim_{x\downarrow 0}\Psi(x)=0$. This gives $V_n^{(q)}(Y,1)\to0$, $\Pp$-a.s.

If $q<p$, then  for every
$L>0$, there exists $\varepsilon\in(0,e^{-1})$ such that
$x^q\ge L\Phi(x)$ for $ 0\le x\le\varepsilon
$. 
  By the
uniform continuity of the sample paths,  the absolute values of all $b$-adic increments 
are bounded by $\varepsilon$ for all sufficiently
large $n$.  Therefore, for all sufficiently large $n$,
\[     V_n^{(q)}(Y,1)     =     \sum_{k=0}^{b^n-1}\abs{Y((k+1)b^{-n})-Y(kb^{-n})}^{q}     \ge     L\sum_{k=0}^{b^n-1}\Phi(\abs{Y((k+1)b^{-n})-Y(kb^{-n})}).
\]
By Theorem~\ref{thm:critical-phi}, the sum on the right converges almost
surely to a strictly positive constant. Since $L>0$ is arbitrary, it follows that $V_n^{(q)}(Y,1)\to\infty$, $\Pp$-a.s. The case $t\in(0,1)$ is similar.\end{proof}

\subsection{Proof of Theorem~\ref{thm:hausdorff-high-hurst}}\label{subsec:proof-hausdorff}

Throughout this section, we assume $H>K$ and $H>1/2$.  The proof of Theorem~\ref{thm:hausdorff-high-hurst} is based on a
lower $L^2$ increment estimate on sets whose $b$-adic orbits avoid the
endpoints of the unit interval, followed by a Frostman energy argument.

Recall that $T_N$ is the set of points $x\in[0,1]$ such that
$\{b^k x\}\in[b^{-N},1-b^{-N}]$ for every $k\ge0$.  The following lemma builds
compact subsets of such sets with Hausdorff dimension close to one.
These Cantor-type sets have dimension arbitrarily close to one and stay uniformly away from the discontinuities of the fractional-part map along all forward $b$-adic iterates. Here and in the sequel, we let  $\N=\{1,2,\ldots\}$.

\begin{lemma}\label{lem:SN}
Let $L\in\N$ be such that $\beta=b^L\ge3$.  Put
$     D_L=\{1,2,\ldots,\beta-2\}
$,
and
\[     S_L=\left\{\sum_{i=1}^{\infty}\xi_i \beta^{-i}:     \xi_i\in D_L\text{ for every }i\ge1\right\}.
\]
Then
\begin{equation}     S_L\subset T_{2L},     \label{eq:SL-subset-T2L}
\end{equation}
and
\begin{equation}     \dimH S_L=\frac{\log(\beta-2)}{\log \beta}.     \label{eq:SL-dimension}
\end{equation}
Moreover, for every
\begin{equation}     0<\eta<\frac{\log(\beta-2)}{\log \beta},     \label{eq:eta-SL-condition}
\end{equation}
there is an atomless Borel probability measure $\mu_L$ supported on $S_L$
such that
\begin{equation}     \iint |t-s|^{-\eta}\,\mu_L(\dd t)\mu_L(\dd s)<\infty.     \label{eq:muL-energy}
\end{equation}
\end{lemma}

\begin{proof}
Each $\xi\in D_L$ has a base-$b$ expansion with exactly $L$ digits, including possibly leading zeros.  The
condition $1\le\xi\le \beta-2$ means that this block of $L$ base-$b$ digits
is neither identically $0$ nor identically $b-1$.

Take $x\in S_L$, and use the base-$b$ expansion obtained by concatenating
the $L$-digit base-$b$ blocks of $\xi_1,\xi_2,\ldots$.  Fix $k\ge0$.
The first $2L$ base-$b$ digits of $\fp{b^k x}$ contain at least one
entire $L$-digit block from this concatenation.  That entire block contains
at least one digit different from $0$ and at least one digit different from
$b-1$.  Therefore, the first $2L$ base-$b$ digits of $\fp{b^k x}$ are
not all zero and are not all $b-1$.  Hence
$     b^{-2L}\le \fp{b^k x}\le1-b^{-2L}
$.
Since $k\ge0$ was arbitrary, $x\in T_{2L}$, proving \eqref{eq:SL-subset-T2L}.

The set $S_L$ is compact.  Indeed, it is the image of the compact space
$D_L^{\N}$ under the continuous map
$     \Psi((\xi_i)_{i\ge1})=\sum_{i=1}^{\infty}\xi_i \beta^{-i}
$.
Moreover, $S_L$ is generated by the similarities
$     \psi_\xi(x)=(x+\xi)/{\beta}$, $\xi\in D_L,
$
in the sense that
$     S_L=\bigcup_{\xi\in D_L}\psi_\xi(S_L)
$.
Set
\[     d_L:=\frac{\log |D_L|}{\log \beta}     =\frac{\log(\beta-2)}{\log \beta}.
\]
The open intervals $\psi_\xi((0,1))$, $\xi\in D_L$, are pairwise disjoint
and contained in $(0,1)$.  Hence the iterated-function system satisfies the
open set condition, with open set $(0,1)$.  The standard dimension theorem
for self-similar sets satisfying the open set condition therefore gives
$\dimH S_L=d_L$; see \cite[Theorem~9.3]{Falconer2014}.  This proves
\eqref{eq:SL-dimension}.

If $\beta=3$, then $D_L=\{1\}$, $S_L$ is a singleton, and $d_L=0$.  In
this case the energy assertion is vacuous, because there is no $\eta$
satisfying \eqref{eq:eta-SL-condition}.  We may therefore assume from now on
that $\beta>3$, equivalently $d_L>0$.

Let $\mu_L$ be the push-forward under $\Psi$ of the product measure on
$D_L^{\N}$ that gives each digit mass $(\beta-2)^{-1}$.  Then $\mu_L$ is a
Borel probability measure supported on $S_L$, and it assigns mass
$(\beta-2)^{-m}$ to every level-$m$ cylinder
$     \psi_{\xi_1}\circ\cdots\circ\psi_{\xi_m}(S_L)$, where $
      \xi_1,\ldots,\xi_m\in D_L
$.
Let $I\subset[0,1]$ be an interval with $0<|I|\le1$, and choose $m\ge0$
such that
$     \beta^{-(m+1)}<|I|\le \beta^{-m}
$.
The interval $I$ intersects at most three level-$m$ base-$\beta$ intervals.
Therefore
\begin{equation}     \mu_L(I)\le 3(\beta-2)^{-m}     =3\beta^{-md_L}     \le 3\beta^{d_L}|I|^{d_L}.     \label{eq:muL-interval-bound}
\end{equation}
The same estimate holds for singleton intervals by decreasing nondegenerate
intervals to the singleton.  Since $d_L>0$, letting $I\downarrow\{x\}$ in
\eqref{eq:muL-interval-bound} gives $\mu_L(\{x\})=0$ for every $x$.  Hence
$\mu_L$ has no atoms.

Now fix $0<\eta<d_L$, and let $X_1,X_2$ be independent random variables
with law $\mu_L$.  Since $\mu_L$ is atomless, we have $\mathbb P(X_1=X_2)=0$ and
\[
\begin{aligned}
\iint |t-s|^{-\eta}\,\mu_L(\dd t)\mu_L(\dd s)
&=\E[\abs{X_1-X_2}^{-\eta}]                                      =\eta\int_0^\infty r^{-\eta-1}\mathbb P(|X_1-X_2|<r)\,\dd r.
\end{aligned}
\]
For $0<r\le1/2$, the interval $(x-r,x+r)\cap[0,1]$ has length at most
$2r$.  Hence \eqref{eq:muL-interval-bound} gives, for every $x\in[0,1]$,
\[     \mu_L((x-r,x+r)\cap[0,1])     \le 3\beta^{d_L}(2r)^{d_L}     =3(2\beta)^{d_L}r^{d_L}.
\]
For $1/2<r\le1$, we use only the trivial bound
$\mu_L((x-r,x+r)\cap[0,1])\le1$.  Since $r^{d_L}\ge2^{-d_L}$, this gives
\[     \mu_L((x-r,x+r)\cap[0,1])     \le 2^{d_L}r^{d_L}     \le 3(2\beta)^{d_L}r^{d_L}.
\]
Thus, with $C_L:=3(2\beta)^{d_L}$, we have for all $0<r\le1$,
\[     \mathbb P(|X_1-X_2|<r)     =\int \mu_L((x-r,x+r)\cap[0,1])\,\mu_L(\dd x)     \le C_L r^{d_L}.
\]
For $r>1$, $\mathbb P(|X_1-X_2|<r)\le1$.  Hence
\[     \eta\int_0^1 r^{-\eta-1}\mathbb P(|X_1-X_2|<r)\,\dd r     \le \eta C_L\int_0^1 r^{d_L-\eta-1}\dd r<\infty,
\]
and
\[     \eta\int_1^\infty r^{-\eta-1}\mathbb P(|X_1-X_2|<r)\,\dd r     \le \eta\int_1^\infty r^{-\eta-1}\dd r=1.
\]
This proves \eqref{eq:muL-energy}.
\end{proof}

The following proposition goes beyond the covariance estimates already available in \cite[Supplement, Lemmas~A.1,~A.2,~and A.3]{SchiedZhang2026}. Lemmas~A.1~and~A.2 show that, in the regimes
$H<K$ and $H=K$, respectively, the increment variance has the 
global lower bounds of Hurst type, namely of order $|t-s|^{2H}$ in the
subcritical case $H<K$ and of order $|t-s|^{2H}\log(1/|t-s|)$ in the critical case $H=K$.
These estimates are sufficient for the Hausdorff-dimension lower bound whenever
the fractional Brownian component is at least as rough as the Weierstrass
component.  In the regime where $     H>K$,
 however, the relevant scale is
$|t-s|^{2K}$, and the lower bound is more delicate because cancellations can
occur across the discontinuities of the fractional-part map.  Therefore, a corresponding lower bound for the increment variance 
can only be obtained on the sets $T_N$, whose $b$-adic orbits
stay uniformly away from these discontinuities. In \cite[Supplement, Lemma~A.3]{SchiedZhang2026}, a corresponding bound was obtained for the case $K\in(2H-1,H)$. 
The following Proposition~\ref{prop:lower-on-TN} now extends this result to the case $     H>K$ and $H>1/2$.

\begin{proposition}\label{prop:lower-on-TN}
Assume $     H>K$ and $H>1/2$.
For every $N\in\N$, there are constants $c_N>0$ and $\delta_N>0$ such
that, for all $s,t\in T_N$ with $0<|t-s|<\delta_N$,
\begin{equation}     \E[\abs{Y(t)-Y(s)}^2]\ge c_N |t-s|^{2K}.     \label{eq:TN-lower-bound}
\end{equation}
\end{proposition}

\begin{proof}
 Assume that \eqref{eq:TN-lower-bound} is false.  Then, for every $n\ge1$, there are
$s_n,t_n\in T_N$ such that
\begin{equation}     0<t_n-s_n<1/n,     \qquad     \E[\abs{Y(t_n)-Y(s_n)}^2]<n^{-1}|t_n-s_n|^{2K}.     \label{eq:contradiction-assumption}
\end{equation}
By passing to a subsequence if necessary, we may assume that $r_n:=t_n-s_n\downarrow0$.  For each
$n$, choose $M_n\in\{0,1,2,\dots\}$ such that $   b^{-M_n-1}<r_n\le b^{-M_n}$. Then $M_n\to\infty$ and $\theta_n:=b^{M_n}r_n\in(1/b,1]$. By passing to another subsequence if necessary, we may furthermore assume that $\theta_n\to\theta$ for some $   \theta\in[1/b,1]$. 

For $j\in\mathbb Z$ with $M_n+j\ge0$, define
$     z_{n,j}=\fp{b^{M_n+j}s_n}$ and     $     y_{n,j}=\fp{b^{M_n+j}t_n}
$.
For $j\in\mathbb Z$ with $M_n+j<0$, $z_{n,j}$ and $y_{n,j}$ will not
be used.  For each  $j\in\mathbb Z$, we have $M_n+j\ge0$ for all sufficiently large $n$, and the compactness of  $[0,1]$ yields  $ z_j,y_j\in[0,1]$ to which $z_{n,j}$ and $y_{n,j}$  converge along a suitable subsequence. By a diagonal subsequence argument, we may assume that, after relabelling of the resulting subsequence,
\begin{equation}
       \lim_{n\uparrow\infty} \theta_n=\theta,\qquad     \lim_{n\uparrow\infty} z_{n,j}= z_j,\qquad
      \lim_{n\uparrow\infty}   y_{n,j}= y_j\qquad\text{for every  $j\in\mathbb Z$.}     \label{eq:diagonal-limits}
\end{equation}

Fix $j\le -N$.  For all sufficiently large $n$, we have $M_n+j\ge0$.  Since
$s_n,t_n\in T_N$, it follows that 
$     z_{n,j},y_{n,j}\in[b^{-N},1-b^{-N}]
$.
Moreover, 
\begin{equation}     b^{M_n+j}t_n=b^{M_n+j}s_n+b^j\theta_n.     \label{eq:scaled-time-relation}
\end{equation}
Because $j\le -N$ and $\theta_n\le1$,
\begin{equation}     0<b^j\theta_n\le b^{-N}.     \label{eq:backward-scale-bound}
\end{equation}
From $z_{n,j}\le1-b^{-N}$ and \eqref{eq:backward-scale-bound},
\begin{equation}     z_{n,j}+b^j\theta_n\le1.     \label{eq:no-wrap-preliminary}
\end{equation}
If equality held in \eqref{eq:no-wrap-preliminary}, then
$     y_{n,j}=\fp{z_{n,j}+b^j\theta_n}=0
$,
contradicting $y_{n,j}\ge b^{-N}$.  Therefore
$     z_{n,j}+b^j\theta_n<1
$,
and \eqref{eq:scaled-time-relation} implies the no-wrap identity
\begin{equation}     y_{n,j}=z_{n,j}+b^j\theta_n.     \label{eq:no-wrap-identity}
\end{equation}
Letting $n\to\infty$ in \eqref{eq:no-wrap-identity} gives
\begin{equation}     y_j=z_j+b^j\theta,\qquad j\le -N.     \label{eq:limit-no-wrap-identity}
\end{equation}

By Lemma~\ref{lem:Y-integrand}, $Y(t_n)-Y(s_n)$ is represented in
$L^2(\Omega)$ by the deterministic integrand
\begin{equation}     g_n(x)=\sum_{m=0}^{\infty}\alpha^m     \left(     J(\fp{b^m s_n},\fp{b^m t_n})(x)     -\kappa(\fp{b^m t_n})     +\kappa(\fp{b^m s_n})     \right),     \label{eq:gn-def}
\end{equation}
and the series in \eqref{eq:gn-def} converges absolutely in $L^p[0,1]$ for $p=1/H$.
For $j\in\mathbb Z$, define
\[     A_{n,j}(x)=     \begin{cases}     J(z_{n,j},y_{n,j})(x)-\kappa(y_{n,j})+\kappa(z_{n,j}),     &M_n+j\ge0,\\     0,&M_n+j<0.     \end{cases}
\]
Using $\theta_n=b^{M_n}r_n$ and that $\alpha=b^{-K}$ for $     H>K$ and $H>1/2$ yields that $\alpha^{M_n+j}=\alpha^j\theta_n^{-K}r_n^K$. 
Then reindexing $m=M_n+j$ in \eqref{eq:gn-def}  gives
\begin{equation}
\begin{aligned}     r_n^{-K}g_n(x)     =\theta_n^{-K}\sum_{j\in\mathbb Z}\alpha^j A_{n,j}(x).
\end{aligned}     \label{eq:gn-reindexed-normalized}
\end{equation}
For $j\in\mathbb Z$, define the limiting terms
$     A_j(x)=J(z_j,y_j)(x)-\kappa(y_j)+\kappa(z_j)
$.
We will now prove that
\begin{equation}     \sum_{j\in\mathbb Z}\alpha^j A_{n,j}     \longrightarrow     \sum_{j\in\mathbb Z}\alpha^j A_j     \quad\text{in }L^p[0,1].     \label{eq:Anj-Lp-convergence}
\end{equation}
First, for $j\ge0$, \eqref{eq:gn-term-Lp-bound} gives
$     \norm{\alpha^j A_{n,j}}_{L^p}     \le C_0\alpha^j$ and 
       $        \norm{\alpha^j A_j}_{L^p}     \le C_0\alpha^j$, where $C_0=1+2\|\kappa\|_\infty$.  Hence
\begin{equation}     \sum_{j\ge R}\sup_n \norm{\alpha^j A_{n,j}}_{L^p}     +\sum_{j\ge R}\norm{\alpha^j A_j}_{L^p}     \le 2C_0\sum_{j\ge R}\alpha^j     \longrightarrow0     \label{eq:positive-tail-bound}
\end{equation}
as $R\to\infty$.
Second, recall that $\kappa$ is $\tau$-H\"older continuous for some $\tau\in(H,1]$. Thus there is a finite constant $C_\kappa$ such that
\begin{equation}     |\kappa(u)-\kappa(v)|\le C_\kappa |u-v|^\tau,     \qquad u,v\in[0,1].     \label{eq:kappa-holder-bound}
\end{equation}
Hence, for $j\le -N$ and $M_n+j\ge0$, \eqref{eq:no-wrap-identity}, Lemma
\ref{lem:J-cont} gives
\begin{equation}
\begin{aligned}     \norm{\alpha^j J(z_{n,j},y_{n,j})}_{L^p}   &=\alpha^j |y_{n,j}-z_{n,j}|^{1/p}     =\alpha^j (b^j\theta_n)^H     \le b^{(H-K)j},                                  \\   \norm{\alpha^j(\kappa(y_{n,j})-\kappa(z_{n,j}))}_{L^p}   &\le C_\kappa\,\alpha^j |y_{n,j}-z_{n,j}|^\tau     \le C_\kappa\, b^{(\tau-K)j}.
\end{aligned}   \label{eq:negative-tail-term-bounds}
\end{equation}
If $M_n+j<0$, then $A_{n,j}=0$, and so 
$\|\alpha^jA_{n,j}\|_{L^p}\le  b^{(H-K)j}+C_\kappa\, b^{(\tau-K)j}$ hold for all $n$.  The same bounds hold for
the limiting terms $A_j$, because of \eqref{eq:limit-no-wrap-identity}.
Since
$   H-K>0$ and $ \tau-K>0$,
we have
$   \sum_{j\le -R} b^{(H-K)j}<\infty$ and 
       $   \sum_{j\le -R} b^{(\tau-K)j}<\infty
$,
and both tails tend to $0$ as $R\to\infty$.  Therefore
\begin{equation}   \sum_{j\le -R}\sup_n \norm{\alpha^j A_{n,j}}_{L^p}   +\sum_{j\le -R}\norm{\alpha^j A_j}_{L^p}   \longrightarrow0   \label{eq:negative-tail-bound}
\end{equation}
as $R\to\infty$.
Third, for fixed $R>N$, the sum over $-R<j<R$ is finite.  For each fixed
such $j$, \eqref{eq:diagonal-limits}, Lemma~\ref{lem:J-cont}, and continuity of $\kappa$
give
\begin{equation}   \norm{A_{n,j}-A_j}_{L^p}\longrightarrow0.   \label{eq:finite-window-convergence}
\end{equation}
Combining the finite convergence \eqref{eq:finite-window-convergence} with the two uniform tail estimates
\eqref{eq:positive-tail-bound} and \eqref{eq:negative-tail-bound} proves \eqref{eq:Anj-Lp-convergence}.

By \eqref{eq:gn-reindexed-normalized}, \eqref{eq:Anj-Lp-convergence}, and $\theta_n^{-K}\to\theta^{-K}$,
\[   r_n^{-K}g_n   \longrightarrow   G:=\theta^{-K}\sum_{j\in\mathbb Z}\alpha^j A_j   \quad\text{in }L^p[0,1].
\]
Lemma~\ref{lem:hilbert} gives convergence in $\calH$:
\begin{equation}   \norm{r_n^{-K}g_n-G}_{\calH}\le C_H   \norm{r_n^{-K}g_n-G}_{L^p}\longrightarrow0.   \label{eq:calH-convergence}
\end{equation}
On the other hand, by \eqref{eq:contradiction-assumption},
\begin{equation}   \norm{r_n^{-K}g_n}_{\calH}^2   =r_n^{-2K}\E[\abs{Y(t_n)-Y(s_n)}^2]   \le n^{-1}\longrightarrow0.   \label{eq:normalized-calH-vanishes}
\end{equation}
Equations \eqref{eq:calH-convergence} and \eqref{eq:normalized-calH-vanishes} imply
$   \norm{G}_{\calH}=0
$.
By the injectivity statement in Lemma~\ref{lem:hilbert},
\begin{equation}   G=0\quad\text{Lebesgue-a.e.~on }[0,1].   \label{eq:G-vanishes-ae}
\end{equation}

It remains to prove that \eqref{eq:G-vanishes-ae} is impossible.  For $j\le -N$, \eqref{eq:limit-no-wrap-identity}
gives
\[   J(z_j,y_j)=\1_{I_j},   \qquad   I_j=[z_j,z_j+b^j\theta],   \qquad   |I_j|=b^j\theta>0,
\]
where the latter inequality follows from the fact that  $\theta\in[1/b,1]$. Let
\begin{equation}   P(x)=\sum_{j\le -N}\alpha^j\1_{I_j}(x).   \label{eq:P-def}
\end{equation}
By \eqref{eq:negative-tail-term-bounds} for the limiting terms,
\[   \sum_{j\le -N}\norm{\alpha^j\1_{I_j}}_{L^p}   \le\sum_{j\le -N} b^{(H-K)j}<\infty.
\]
Thus \eqref{eq:P-def} converges in $L^p$, and after changing $P$ on a null set,
\begin{equation}   P(x)=\sum_{j\le -N}\alpha^j\1_{I_j}(x)   \quad\text{for every }x\text{ outside a null set }E.   \label{eq:P-pointwise-version}
\end{equation}
For $x\notin E$, all summands in \eqref{eq:P-def} are nonnegative.  Hence, $P(x)\ge \alpha^j$ for every 
fixed $ x\in I_j\setminus E$ and $j\le -N$. Define the remainder
\begin{equation}
\begin{aligned}   R(x)   &=   \sum_{j>-N}\alpha^j J(z_j,y_j)(x)   +\sum_{j\in\mathbb Z}\alpha^j     \big(-\kappa(y_j)+\kappa(z_j)\big).
\end{aligned}   \label{eq:R-def}
\end{equation}
Then
\begin{equation}   G(x)=\theta^{-K}\big(P(x)+R(x)\big)   \quad\text{a.e.}   \label{eq:G-P-R-decomposition}
\end{equation}
For the signed indicators in \eqref{eq:R-def}, we have 
$J(z_j,y_j)(x)\ge -1$ for $ j>-N$,
and therefore
\begin{equation}   \sum_{j>-N}\alpha^j J(z_j,y_j)(x)   \ge -\sum_{j>-N}\alpha^j   >-\infty.   \label{eq:signed-indicator-lower-bound}
\end{equation}
For the bridge constants with $j>-N$,
\begin{equation*}   \sum_{j>-N}\alpha^j   |-\kappa(y_j)+\kappa(z_j)|   \le 2\|\kappa\|_\infty\sum_{j>-N}\alpha^j   <\infty.   
\end{equation*}
For the bridge constants with $j\le -N$, \eqref{eq:limit-no-wrap-identity} and \eqref{eq:kappa-holder-bound} give
\begin{equation}
   \alpha^j|\kappa(y_j)-\kappa(z_j)|   \le C_\kappa\,\alpha^j (b^j\theta)^\tau   \le C_\kappa\,b^{(\tau-K)j}\quad\text{and}\quad     \sum_{j\le -N}b^{(\tau-K)j}<\infty.   \label{eq:negative-bridge-constant-sum}
\end{equation}
Equations \eqref{eq:signed-indicator-lower-bound}--\eqref{eq:negative-bridge-constant-sum} imply that there is a finite constant $C_N$ such
that
\begin{equation}   R(x)\ge -C_N   \quad\text{for Lebesgue-a.e.~}x\in[0,1].   \label{eq:R-lower-bound}
\end{equation}
Choose $j_0\le -N$ so negative that
$   \alpha^{j_0}>C_N
$.
 Let $F$ be the full-measure set on which \eqref{eq:P-pointwise-version}, \eqref{eq:G-P-R-decomposition},
and \eqref{eq:R-lower-bound} hold.  Since $|I_{j_0}|>0$, the set $I_{j_0}\cap F$ has
positive Lebesgue measure.  On that set,
\[   G(x)=\theta^{-K}(P(x)+R(x))   \ge \theta^{-K}(\alpha^{j_0}-C_N)>0.
\]
Therefore, $G\ne0$ on a set of positive Lebesgue measure, contradicting
\eqref{eq:G-vanishes-ae}.  The contradiction proves the proposition.
\end{proof}

\begin{proof}[Proof of Theorem~\ref{thm:hausdorff-high-hurst}] The proof is now analogous to the proof of \cite[Theorem 2.6]{SchiedZhang2026}. Specifically, the upper bound for the Hausdorff dimension follows from the H\"older regularity of the sample paths of $Y$ \cite[Proposition~A.1]{SchiedZhang2024} combined with the standard covering estimate for graphs of H\"older
functions in 
\cite[Section~11.1]{Falconer2014}.
    The lower bound is obtained by a potential-theoretic lower bound in complete analogy to the case $K<H$ in the proof of \cite[Theorem 2.6]{SchiedZhang2026}, using the energy bound \eqref{eq:muL-energy} and replacing \cite[Supplement,~Lemma~A.3]{SchiedZhang2026} with Proposition~\ref{prop:lower-on-TN}.
\end{proof}

\appendix
\section{Appendix}\label{appendix}

In this appendix, we collect several auxiliary facts on fractional Brownian motion.
Let $\mathcal H_H$ be the canonical Gaussian Hilbert space of the fractional
Brownian motion $W_H$, that is, the completion of the linear span of
$\{\1_{[0,t]}:0\le t\le1\}$ under the inner product
\[   \big\langle \1_{[0,s]},\1_{[0,t]}\big\rangle_{\mathcal H_H}        =\E[W_H(s)W_H(t)]   =\frac12\big(s^{2H}+t^{2H}-|t-s|^{2H}\big)   =:R_H(s,t).
\]
Equivalently, $W_H$ is an isonormal Gaussian process over $\calH$.  Thus, for
$f\in\calH$, we write
\[   \int_0^1 f\,\dd W_H := W_H(f),
\]
and for all $f,g\in\calH$,
\begin{equation*}   \E\left[   \left(\int_0^1 f\,\dd W_H\right)   \left(\int_0^1 g\,\dd W_H\right)   \right]   =   \left\langle f,g\right\rangle_{\calH}.
\end{equation*}
This is the deterministic Wiener integral with respect to fractional Brownian
motion.  The notation $\int_0^1 f\,\dd W_H$ is understood in this Hilbert-space
sense. The space $\mathcal H_H$ is separable, since the linear span of
$\{\1_{[0,r]}:r\in[0,1]\cap\mathbb Q\}$ is dense.

For the  proof of Theorem \ref{thm:hausdorff-high-hurst}, we need  the concrete
deterministic-integrand realization of the canonical Hilbert space for $H>1/2$.

\begin{lemma}\label{lem:hilbert}Assume $H>1/2$, and put $p=1/H$.  For step functions $f,g$ on
$[0,1]$, the inner product on $\calH$ has the concrete
representation
\begin{equation}   \left\langle f,g\right\rangle_{\calH}   =   H(2H-1)\int_0^1\int_0^1   f(x)g(y)|x-y|^{2H-2}\,\dd x\,\dd y .   \label{eq:calH-inner-product}
\end{equation}
Moreover, $L^p[0,1]$ embeds continuously into $\calH$: there is a finite
constant $C_H$ such that
\begin{equation}   \norm{f}_{\calH}   \le C_H\norm{f}_{L^p[0,1]},   \qquad f\in L^p[0,1].   \label{eq:calH-Lp-bound}
\end{equation}
Under this embedding, \eqref{eq:calH-inner-product} remains valid for all
$f,g\in L^p[0,1]$.  Finally,
\begin{equation*}   \norm{f}_{\calH}=0,\quad f\in L^p[0,1],   \quad\Longrightarrow\quad   f=0\text{ Lebesgue-a.e.}
\end{equation*}
In particular, the embedding $L^p[0,1]\hookrightarrow\calH$ is injective.
\end{lemma}

\begin{proof}
For $H>1/2$,
\[   \left\langle \1_{[0,u]},\1_{[0,v]}\right\rangle_{\calH}   =   R_H(u,v)=  H(2H-1)\int_0^u\int_0^v |x-y|^{2H-2}\,\dd x\dd y.
\]
By bilinearity, this proves \eqref{eq:calH-inner-product} for step functions.

The estimate \eqref{eq:calH-Lp-bound} follows from the one-dimensional
Hardy--Littlewood--Sobolev inequality, equivalently the fractional-integral
estimate recalled in \cite[Corollary~1.9.2]{Mishura2008}; see also
\cite[Chapters~V and~XIII]{SamkoKilbasMarichev1993}.  Since step functions
are dense in $L^p[0,1]$, this estimate extends the identity map on step
functions to a continuous linear map $L^p[0,1]\to\calH$.  By continuity,
\eqref{eq:calH-inner-product} remains valid for $f,g\in L^p[0,1]$.

It remains to prove injectivity.  After extending $f\in L^p[0,1]$ by zero
to $\R$, the Riesz-kernel Plancherel formula gives
\[   \norm{f}_{\calH}^2   =   c_H\int_{\R}|\widehat f(\xi)|^2|\xi|^{1-2H}\,\dd \xi,   \qquad c_H>0,
\]
first for smooth compactly supported functions and then by approximation using
\eqref{eq:calH-Lp-bound}.  If $\norm{f}_{\calH}=0$, then
$\widehat f(\xi)=0$ for Lebesgue-a.e.~$\xi\ne0$.  Since $p>1$ and
$[0,1]$ has finite measure, $f\in L^1(\R)$, so $\widehat f$ is
continuous.  Hence $\widehat f\equiv0$, and uniqueness of the Fourier
transform yields $f=0$ Lebesgue-a.e.
\end{proof}

For $a,c\in[0,1]$, put
\begin{equation*}   J(a,c):=   \begin{cases}   \1_{[a,c]},&a\le c,\\   -\1_{[c,a]},&a>c.   \end{cases}
\end{equation*}
With this deterministic-integral convention, bridge increments have the following immediate representation as Wiener integrals with explicitly oriented interval indicators:
\begin{equation}   B_H(c)-B_H(a)   =\int_0^1\left(J(a,c)-(\kappa(c)-\kappa(a))\right)\dd W_H .   \label{eq:bridge-increment-integral}
\end{equation}
Here the scalar $\kappa(c)-\kappa(a)$ is identified with the constant
function on $[0,1]$.
Summing the preceding representation over Weierstrass scales gives an integral representation for increments of $Y$.

\begin{lemma}\label{lem:Y-integrand}
Assume $H>1/2$ and put $p=1/H$.  For fixed $s,t\in[0,1]$, define
\[
\begin{aligned}   h_{s,t}(x):=\sum_{m=0}^{\infty}\alpha^m   \left(   J(\fp{b^m s},\fp{b^m t})(x)   -\kappa(\fp{b^m t})+\kappa(\fp{b^m s})   \right).
\end{aligned}
\]
Then the series defining $h_{s,t}$ converges absolutely in $L^p[0,1]$, and
\[   Y(t)-Y(s)=\int_0^1 h_{s,t}\dd W_H   \quad\text{in }L^2(\Omega).
\]
Consequently $Y(t)-Y(s)$ is a centered Gaussian random variable and
\[   \operatorname{Var}(Y(t)-Y(s))=\|h_{s,t}\|_{\calH}^2.
\]
\end{lemma}

\begin{proof}
The bound
\begin{equation}
 \left\|   J(\fp{b^m s},\fp{b^m t})   -\kappa(\fp{b^m t})+\kappa(\fp{b^m s})
 \right\|_{L^p}
 \le 1+2\|\kappa\|_\infty
 \label{eq:gn-term-Lp-bound}
\end{equation}
and the summability of $\sum_m\alpha^m$ give absolute convergence of the
series for $h_{s,t}$ in $L^p[0,1]$.  For the partial sums
$h^{(q)}_{s,t}$, \eqref{eq:bridge-increment-integral} gives
\begin{equation*}   \sum_{m=0}^{q}\alpha^m   \big(B_H(\fp{b^m t})-B_H(\fp{b^m s})\big)   =\int_0^1 h^{(q)}_{s,t}\,\dd W_H .
\end{equation*}
The left-hand side converges in $L^2(\Omega)$ by
$\sup_{0\le u\le1}\|B_H(u)\|_{L^2}<\infty$, while Lemma~\ref{lem:hilbert}
turns the $L^p$-convergence of $h^{(q)}_{s,t}$ into convergence of the
right-hand side in $L^2(\Omega)$.  This proves the integral representation;
the Gaussianity and variance identity follow from the deterministic-integrand
isometry.
\end{proof}

We shall also use the following elementary continuity property of oriented interval indicators.

\begin{lemma}\label{lem:J-cont}
Let $1\le q<\infty$.  If $a_n\to a$ and $c_n\to c$ in $[0,1]$, then
\begin{equation}   \norm{J(a_n,c_n)-J(a,c)}_{L^q[0,1]}\longrightarrow0.   \label{eq:oriented-indicator-continuity}
\end{equation}
Moreover,
\begin{equation}   \norm{J(a,c)}_{L^q[0,1]}=|c-a|^{1/q},   \qquad   \norm{J(a,c)}_{L^\infty[0,1]}\le1.   \label{eq:oriented-indicator-norms}
\end{equation}
\end{lemma}

\begin{proof}
The norm identities in \eqref{eq:oriented-indicator-norms} are immediate from the definition of $J$.  For
\eqref{eq:oriented-indicator-continuity}, decompose the oriented indicators into ordinary interval indicators.
Their symmetric difference is contained in intervals whose total length tends to
zero when the two endpoints converge.  Hence the $L^q$-norm of the difference
converges to zero for every finite $q$.
\end{proof}

For intervals $I=[u,u+L]$ and $J=[v,v+\ell]$, $0<\ell\le L$, define the
normalized fractional Brownian covariance
\[   \Gamma(I,J):=L^{-H}\ell^{-H} \Cov(\Delta W_H(I),\Delta W_H(J)),
\]
where we let $\Delta W_H(I):=W_H(u+L)-W_H(u)$ and define $\Delta W_H(J)$ analogously.
Put
\[   \theta:=H\wedge(1-H).
\]
We will need the following two-scale covariance bound for fractional Brownian increments.

\begin{lemma}\label{lem:fbm-two-scale-covariance}
There is a constant $C$, depending only on $H$, such that for all
intervals $I,J\subset[0,1]$ with lengths $L\ge \ell$,
\begin{equation}   \abs{\Gamma(I,J)}   \le C\left(\frac{\ell}{L}\right)^\theta .   \label{eq:two-scale-bound}
\end{equation}
If $I$ and $J$ are disjoint and their distance is at least $d>0$, then
\begin{equation}   \abs{\Gamma(I,J)}   \le C L^{1-H}\ell^{1-H}d^{2H-2}.                              \label{eq:separated-bound}
\end{equation}
\end{lemma}

\begin{proof}
For \eqref{eq:separated-bound}, suppose that $u+L\le v$.
If $H=1/2$, the covariance is zero.  If $H\ne1/2$, the covariance
identity gives
\[
 \Cov(\Delta W_H(I),\Delta W_H(J))
 =
 H(2H-1)\int_u^{u+L}\int_v^{v+\ell}\abs{y-x}^{2H-2}\dd y\dd x
\]
where the constant $H(2H-1)$ is negative for $H<1/2$.  Since the
intervals are disjoint, taking absolute values yields
\[   \abs{\Cov(\Delta W_H(I),\Delta W_H(J))}   \le C L\ell\,d^{2H-2}.
\]
After division by $L^H\ell^H$, this proves \eqref{eq:separated-bound}.

It remains to prove \eqref{eq:two-scale-bound}.  If the distance between
$I$ and $J$ is at least $L$, then \eqref{eq:separated-bound} gives
\[   \abs{\Gamma(I,J)}   \le C\left(\frac{\ell}{L}\right)^{1-H}   \le C\left(\frac{\ell}{L}\right)^\theta .
\]
If the two intervals are closer than $L$, then after translating and scaling
by $L$ we may write the longer interval as $[a,a+1]$, the shorter one as
$[c,c+q]$, where $q=\ell/L$, and all endpoints lie in a fixed bounded
interval, say $[-2,3]$.  With $f(x)=|x|^{2H}$, the scaled covariance is
\begin{align*}
q^H|\Gamma(I,J)|&=L^{-2H}|\Cov(\Delta W_H(I),\Delta W_H(J))|\\
&=\frac12\Big(|a+1-c|^{2H}-|a+1-c-q|^{2H}-|a-c|^{2H}+|a-c-q|^{2H}\Big).
\end{align*}
This is a sum of two increments of $f(x):=|x|^{2H}$, both over intervals of length $q$
and with arguments in a fixed bounded set.  On that set, $f$ is H\"older
continuous with exponent $2H$ if $H\le1/2$, and Lipschitz if $H>1/2$.
Hence,
\[q^H|\Gamma(I,J)|\le  \begin{cases}Cq^{2H}&\text{if $H\le1/2$,}\\
Cq&\text{if  $H>1/2$.  }
\end{cases}
\]
Therefore, $|\Gamma(I,J)|$ is bounded by 
 $Cq^H$ for $H\le1/2$ and by $Cq^{1-H}$ for  $H>1/2$.
This  proves~\eqref{eq:two-scale-bound}.
\end{proof}

\begingroup
\raggedright
\bibliographystyle{plain}
\bibliography{ww_references}
\endgroup

\end{document}